\RequirePackage{fix-cm}
\documentclass[ijoc,nonblindrev]{informs3_hide}

\usepackage[T1]{fontenc}
\usepackage{mathtools}
\IfFileExists{tgtermes.sty}{\usepackage{tgtermes}}{}
\IfFileExists{newtxmath.sty}{\usepackage{newtxmath}}{}
\usepackage{bm}
\usepackage{endnotes}
\OneAndAHalfSpacedXII

\IfFileExists{algorithm.sty}{\usepackage{algorithm}}{%
  \usepackage{float}\floatstyle{ruled}
  \newfloat{algorithm}{tbp}{loa}\floatname{algorithm}{Algorithm}}
\IfFileExists{algpseudocode.sty}{\usepackage{algpseudocode}}{}
\usepackage{tikz}
\usepackage{booktabs}
\usepackage{placeins}
\usepackage{natbib}
\bibpunct[, ]{(}{)}{,}{a}{}{,}

\usepackage{xurl}
\usepackage[hidelinks]{hyperref}
\usepackage[capitalize,noabbrev]{cleveref}

\EquationsNumberedThrough
\TheoremsNumberedThrough
\ECRepeatTheorems
\newcommand{\I}{\mathbf{1}}

\MANUSCRIPTNO{}

\begin{document}
\RUNAUTHOR{Fu, Jiang, Wang and Zhou}
\RUNTITLE{Contextual First-Price Auctions with Budgets}
\TITLE{Online Bidding for Contextual First-Price Auctions with Budgets under One-Sided Information Feedback}
\ARTICLEAUTHORS{%
\AUTHOR{$\text{Zeng Fu}^{\dag}$,~~$\text{Jiashuo Jiang}^{\ddag}$, ~~$\text{Yige Wang}^{\ddag}$,~~$\text{Yuan Zhou}^{\S\P}$}
\AFF{
$\dag~$~Qiuzhen College, Tsinghua University\\
$\ddag$~Department of Industrial Engineering and Decision Analytics, Hong Kong University of Science and Technology\\
$\S$~Yau Mathematical Sciences Center \& Department of Mathematical Sciences, Tsinghua University\\
$\P$~Beijing Institute of Mathematical Sciences and Applications
}
}
\ABSTRACT{%
In this paper, we study learning to bid in repeated first-price auctions with a hard budget constraint. We focus on one-sided information feedback: the highest competing bid is observed after a loss but remains unobserved after a win. The competing bid depends on the impression context through an unknown coefficient and additive noise with an unknown distribution. To learn this model from censored observations, we propose a regression method based on conditional quantile invariance. Combining this estimator with bid elimination and a dual update procedure, we obtain expected regret growing at most as the square root of the horizon, up to logarithmic factors, against the optimal hard-budget policy under the stated regularity, identifiability, visibility, and effective-value conditions. This matches the horizon dependence of established lower bounds for general first-price auction learning.
We also discuss a structured multi-dimensional extension in which contexts lie on a monotone one-dimensional manifold.
}
\maketitle

\section{Introduction}
\label{sec:introduction}

Digital advertising has undergone a broad transition from second-price to
first-price auctions~\citep{despotakis2019first}. This shift fundamentally
alters bidder incentives. In a first-price auction, truthful bidding is no
longer optimal: bidders must shade their bids below their private values,
balancing the probability of winning against the surplus retained after
payment. A finite budget further couples these decisions over time, since
spending on the current impression leaves less budget for future auctions.

Learning how to bid is particularly challenging when competing bids depend
on the impression context and the highest competing bid is revealed only
after a loss. Which competing bids are observed depends on the bidder's own
decisions, so directly fitting a model to the observed samples can introduce
selection bias. We study repeated first-price auctions over a horizon of $T$ rounds with this
one-sided feedback, budget constraints, contextual competitors, and an unknown
noise distribution.

To learn the contextual effect from censored observations, we develop a
regression method based on conditional quantile invariance. Combined with
bid elimination and a dual update for controlling expenditure, this estimator
leads to $\widetilde O(\sqrt T)$ expected regret against the optimal
hard-budget policy under the assumptions below. We also discuss a structured
multi-dimensional extension in which contexts lie on a monotone
one-dimensional manifold.

The square-root horizon order is a classical benchmark in first-price
auction learning. \citet{han2025optimal} establish an $\Omega(\sqrt T)$
worst-case lower bound even with a constant private value and observation
of the highest competing bid after every auction, while its distribution
remains unknown. Our result attains the same dependence on $T$, up to
logarithmic factors, under the assumptions of our contextual model.
The contribution is to retain this horizon order while jointly learning
the contextual coefficient and noise distribution from one-sided feedback
and enforcing a hard budget. The analysis must account for how these
estimation errors interact with adaptive bidding and budget control.

\subsection{Problem formulation}
\label{sec:problem-formulation}

Consider a bidder participating in repeated first-price auctions over a
horizon $T$ with total budget $B=\rho T$, where $\bar v$ is the known upper
bound on private values and $0<\rho\le\bar v$. The latter
restriction is without loss for the budgeted problem, since the budget is
nonbinding when $\rho>\bar v$. At round $t$, the bidder observes a
context $X_t=x_t$ and a private value
$V_t=v_t=f(x_t)$ before submitting a bid $b_t\in[0,v_t]$. Let $D_t=d_t$ be the
highest competing bid. Write $\I\{E\}$ for the indicator of an event $E$,
$\Pi_C$ for Euclidean projection onto a closed convex set $C$, and
$a\vee b=\max\{a,b\}$. The bidder wins when $b_t>d_t$, receives reward
$(v_t-b_t)\I\{b_t>d_t\}$, and pays $b_t\I\{b_t>d_t\}$.

We consider one-sided information feedback. After a losing round
$b_t\le d_t$, the highest competing bid $d_t$ is observed. After a winning
round $b_t>d_t$, the bidder observes the winning outcome and her own payment,
but not $d_t$. The competing bid follows the contextual model
\begin{equation}
\label{eq:competing-bid-model}
D_t=\alpha X_t+Z_t,
\end{equation}
where the coefficient $\alpha$ and the distribution function $G$ of $Z_t$ are
unknown. Conditional on $(X_t,V_t)$ and the submitted bid, the expected reward
and expenditure are therefore
\begin{equation}
\label{eq:reward-cost-functions}
r(v,b,x):=(v-b)G(b-\alpha x),
\qquad
c(b,x):=bG(b-\alpha x).
\end{equation}
A policy is nonanticipating if each bid uses only the current context and
value and the information available from previous rounds; the policy may
also randomize. Let $\Pi_B$ denote the class of such policies satisfying the
hard budget constraint $\sum_{t=1}^T b_t^\pi\I\{b_t^\pi>d_t\}\le B$ almost surely. For $\pi\in\Pi_B$, define $R(\pi) := \mathbb E^\pi\![ \sum_{t=1}^T (V_t-b_t^\pi)\I\{b_t^\pi>D_t\} ]$. Let $\pi^\star\in\arg\max_{\pi\in\Pi_B}R(\pi)$ be an optimal hard-budget
policy. The expected regret of a policy $\pi$ is
\begin{equation}
\label{eq:regret-definition}
\operatorname{Regret}(\pi)
:=
R(\pi^\star)-R(\pi).
\end{equation}
Throughout the paper, we impose the following assumptions. In all theoretical
statements involving quantile balancing, we fix the quantile level
$p_0:=0.99$.

\begin{assumption}[Regularity of the noise distribution]
\label{ass:lipschitz}
Let $G$ be the cumulative distribution function of $Z$ and let $g$ be its
density. There exists a constant $g_{\max}<\infty$ such that
$0\le g(z)\le g_{\max}$ on the support of $Z$. Use the generalized inverse $G^{-1}(p):=\inf\{z:G(z)\ge p\}$ and
write $q_0:=G^{-1}(p_0)$.
There is a radius $r_g>0$ such that $g$ is Lipschitz on
$[q_0-r_g,q_0+r_g]$ and $g(z)\ge g_{\min}>0$ for every
$z\in[q_0-r_g,q_0+r_g]$. The radius $r_g$ records the neighborhood on which this local lower bound
is available.
In particular, $G$ is $g_{\max}$-Lipschitz.
\end{assumption}

\begin{assumption}[Boundedness]
\label{ass:bounded}
The contexts, values, noise terms, and unknown contextual coefficient satisfy $X_t\in[0,\bar x]$, $V_t=f(X_t)\in[0,\bar v]$, $Z_t\in[0,\bar z]$, and $\alpha\in[0,\bar\alpha]$. All bounds are known, and admissible bids belong to $[0,\bar v]$. The
sequences $\{X_t\}_{t=1}^T$ and $\{Z_t\}_{t=1}^T$ are i.i.d. and mutually
independent.
\end{assumption}

\begin{assumption}[Local quantile identifiability]
\label{ass:identifiability}
Let $m_X$ be a population median of $X$, and define
$J_1:=\{x:x\le m_X\}$ and $J_2:=\{x:x>m_X\}$.
Assume that $X$ is atomless, so that both groups have probability $1/2$.
For any candidate parameter $\widetilde\alpha$, let
$q_k^o(\widetilde\alpha)$ be the population $p_0$-quantile of $D-\widetilde\alpha X = Z+(\alpha-\widetilde\alpha)X$ conditional on $X\in J_k$, for $k=1,2$. There exist constants
$\kappa>0$ and $r_{\mathrm{id}}>0$ such that $| q_1^o(\widetilde\alpha) - q_2^o(\widetilde\alpha) | \ge \kappa|\widetilde\alpha-\alpha|$ whenever $|\widetilde\alpha-\alpha|\le r_{\mathrm{id}}$.
\end{assumption}

\begin{assumption}[Visibility of high competing bids]
\label{ass:visible-quantile}
There exists a constant $\Delta_q>0$ such that, for every
$x\in\operatorname{supp}(X)$,
$\mathbb P(D_t\ge f(x)+\Delta_q\mid X_t=x)\ge1-p_0$.
Equivalently, under $D_t=\alpha X_t+Z_t$,
$\mathbb P(Z_t\ge f(x)-\alpha x+\Delta_q)\ge1-p_0$ for every
$x\in\operatorname{supp}(X)$. Consequently,
$G^{-1}(p_0)\ge f(x)-\alpha x+\Delta_q$ throughout this support.
\end{assumption}

\begin{assumption}[Effective-value margin]
\label{ass:effective-value-margin}
There exists a constant $\gamma_q\in(0,1]$ such that
$f(x)/(1+\bar\lambda)-\bar\alpha x\ge\gamma_qf(x)$ for every
$x\in\operatorname{supp}(X)$, where $\bar\lambda=\bar v/\rho$.
\end{assumption}

The constant $\bar\lambda$ is an upper bound for the Lagrange multiplier
used in Section~\ref{sec:algorithm} to control expenditure.

\begin{assumption}[Value-density regularity]
\label{ass:value-density}
The random value $V=f(X)$ has support $[0,\bar v]$ and admits a density
$h_V$ satisfying $0<\underline h_V\le h_V(v)\le\overline h_V<\infty$
for every $v\in[0,\bar v]$.
\end{assumption}

The above assumptions combine regularity conditions with structural
conditions specific to contextual auctions under one-sided feedback.
Assumptions~\ref{ass:lipschitz} and~\ref{ass:bounded} ensure stable quantiles,
bounded observations, and independence between the context and noise.

\label{app:assumption-discussion}
\paragraph{Discussion of Assumption~\ref{ass:identifiability}.}
The assumption identifies $\alpha$ locally by comparing residual quantiles
in the ordered population median groups $J_1,J_2$. At the true coefficient,
independence of $Z$ and $X$ gives $q_1^o(\alpha)=q_2^o(\alpha)=q_0$,
where $q_0=G^{-1}(p_0)$. For bounded $X$ and a continuous positive density
near $q_0$, expanding the conditional residual CDF
$\mathbb E[G(y-(\alpha-\widetilde\alpha)X)\mid X\in J_k]$
and locally inverting it yields
$q_k^o(\widetilde\alpha)=q_0+(\alpha-\widetilde\alpha)\mathbb E[X\mid X\in J_k]+o(|\widetilde\alpha-\alpha|)$, for $k=1,2$;
see \citet[Chapter~21]{vandervaart1998asymptotic}.
The upper group's conditional mean is strictly larger, so the quantile
gap has a nonzero linear term. For $X\sim\mathcal U(0,1)$, the conditional
means are $1/4$ and $3/4$, and any fixed $0<\kappa<1/2$ is valid on a
sufficiently small neighborhood. Neither $\kappa$ nor the identification
radius is uniform across models; both enter the finite-sample bounds.
The vector case uses Assumption~\ref{ass:multi-identifiability}.

\paragraph{Discussion of Assumption~\ref{ass:visible-quantile}.}
The learner observes $D_t$ after a loss, whereas a win reveals only
$D_t<b_t$. Since $b_t\le f(X_t)$, every censored residual satisfies $D_t-\alpha X_t<f(X_t)-\alpha X_t$. The visibility assumption places the target noise quantile above this
entire region by a uniform positive margin, $G^{-1}(p_0) \ge f(x)-\alpha x+\Delta_q,  x\in\operatorname{supp}(X).$
On the concentration event, this separation permits replacement by
$-\infty$ without changing empirical quantiles. Localization controls
the residual shift from replacing $\alpha$.

The condition requires probability at least
$1-p_0=0.01$ that the competing bid exceeds the bidder's value by
$\Delta_q$ at every context. A right tail in the competing-bid distribution
alone does not imply this uniform separation.
For instance, if $Z\sim\mathcal U(0,\bar z)$, the condition holds when
$p_0\bar z \ge\sup_{x\in\operatorname{supp}(X)}\{f(x)-\alpha x\}+\Delta_q.$
A bounded truncated normal is also admissible if its $p_0$-quantile
satisfies the same separation bound; boundedness or smoothness alone is insufficient.

The algorithm uses the fixed level $p_0=0.99$, without knowing
$G^{-1}(p_0)$ or $\Delta_q$. Finite zero-bid samples cannot certify uniform visibility over the
context support; the guarantee requires this separation.

Assumption~\ref{ass:effective-value-margin} ensures that, after adjusting the
value for the budget multiplier and subtracting the contextual component,
a fixed positive fraction of the private value remains.
Assumption~\ref{ass:value-density} excludes gaps and excessive concentration
in the value distribution; the density need not be known. For example, this
condition holds when $X$ has a density bounded above and away from zero and
$f$ maps its support onto $[0,\bar v]$, with derivative bounded above and
away from zero. Together, these two assumptions allow us to control the
number of informative observations available for each shifted bid as the
bidding rule and budget multiplier change.

\subsection{Main results and contributions}

Our main contribution is a bidding algorithm for contextual first-price
auctions with budget constraints (Algorithm~\ref{alg:budget-bidding}) that
achieves $\widetilde O(\sqrt T)$ regret
(Theorem~\ref{thm:main-regret}). The contributions are threefold.

\begin{enumerate}
    \item \textbf{Contextual bidding with budgets.}
    Earlier $\widetilde O(\sqrt T)$ guarantees under one-sided feedback
    treat context-independent competing bids, without budgets
    in~\citet{han2025optimal} and with budgets in~\citet{wang2023learning}.
    We allow competition to vary with the observed context while enforcing
    the hard budget constraint.
    \item \textbf{Learning an unknown competing-bid model.}
    Our algorithm learns both the contextual coefficient $\alpha$ and the
    noise distribution $G$ under one-sided feedback. Contextual bidding with
    an unknown noise distribution has previously been studied under full
    information, as well as under binary feedback for a parametric noise
    family~\citep{badanidiyuru2023learning}. We combine joint estimation
    under one-sided censoring with budget control.
    \item \textbf{Estimation from censored observations.}
    We develop a regression method based on conditional quantile invariance
    (Algorithm~\ref{alg:quantile}) and prove an estimation error bound of
    $\widetilde O(|A_i|^{-1/2})$, where $|A_i|$ is the number of observations
    used for parameter estimation in phase $i$. Under the stated visibility
    and local identifiability conditions, the method handles the adaptive
    one-sided censoring induced by the bidding rule.
\end{enumerate}

\subsection{Related work}

\paragraph{Learning in repeated auctions.}
Classical auction theory studies bidding incentives, equilibrium behavior,
and mechanism design~\citep{wilson1969communications,myerson1981optimal,
riley1981optimal,wilson1985game}. Online-learning formulations instead
construct bidding or reserve-price rules from observed outcomes; examples
include reserve-price learning in second-price
auctions~\citep{mohri2014learning,cesa2014regret,roughgarden2019minimizing}
and bidder-side learning~\citep{weed2016online}.
For first-price auctions with an unknown i.i.d. competing-bid distribution,
\citet{han2025optimal} obtain $\widetilde O(\sqrt T)$ regret under one-sided
feedback, for both stochastic and adversarial private values.
\citet{han2020learning} allow adversarial competing bids and obtain a
square-root-order guarantee under full-information feedback against
Lipschitz bidding policies. These results establish the role of the
feedback structure and the comparison policy in auction learning.

\paragraph{Bidding with budgets.}
\citet{balseiro2019learning} study budget pacing and regret minimization in
repeated auctions. For first-price auctions with stationary competition,
\citet{ai2022no} use a dynamic-programming-based approach and obtain
$\widetilde O(\sqrt T)$ regret under full information and
$\widetilde O(T^{7/12})$ under censored feedback.
\citet{wang2023learning} combine dual updates with bidding algorithms and
obtain $\widetilde O(\sqrt T)$ regret under both full and one-sided feedback,
with additional value-distribution conditions in the latter case.
Our algorithm builds on this dual and bid-elimination approach while
learning the contextual dependence of competing bids.

\paragraph{Contextual competition.}
\citet{badanidiyuru2023learning} study a linear contextual competing-bid
model with log-concave noise. They consider binary feedback with a known
noise distribution, a parametric unknown-noise extension under binary
feedback, and full-information feedback with an unknown noise distribution.
Thus, learning both a contextual coefficient and a noise distribution has
precedents; our setting combines this task with one-sided feedback and a
hard budget under different structural and regularity assumptions.
\citet{balseiro2023contextual} study equilibrium bidding in contextual
standard auctions with budget constraints in expectation. Their equilibrium
and revenue-equivalence analysis addresses a different question from the
learning performance of a single bidder considered here.

\paragraph{Resource-constrained online learning.}
Budgeted bidding is related to bandits with
knapsacks~\citep{badanidiyuru2018bandits} and their contextual
extensions~\citep{agrawal2016efficient}. Dual mirror descent provides a
general approach to online resource allocation~\citep{balseiro2022best}.
The comparison benchmark also matters: \citet{chen2024contextual} show
that the fluid benchmark, which imposes resource constraints only in
expectation, can exceed the optimal online reward by $\Omega(\sqrt T)$,
and obtain faster rates under additional nondegeneracy conditions.
Here regret is defined relative to the optimal online hard-budget policy;
the expected-budget relaxation serves as an upper bound in the analysis.

\section{Quantile-based Estimation under Censored Feedback}
\label{sec:robust-regression}
The core statistical challenge is estimating $\alpha$ from one-sided
censored observations. Regressing on revealed losses alone generally
introduces selection bias because the censoring depends on adaptive bids.
We write $A_i$ for the estimation rounds in phase $i$; the phase schedule
is specified in Section~\ref{sec:algorithm}.

\medskip
\noindent
\textbf{Quantile-based Estimator.}
The local expansion in Section~\ref{sec:problem-formulation} explains why
balancing residual quantiles across the ordered context groups identifies
$\alpha$ near its true value. Assumption~\ref{ass:visible-quantile} places
the target quantile above the residuals that can be censored.
Lemma~\ref{lem:quantile-invariance} shows that replacement by $-\infty$
preserves the empirical quantiles uniformly over a local parameter interval
with high probability, so this comparison can use every round.

For $y_1,\ldots,y_N\in\mathbb R\cup\{-\infty\}$, write
$y_{(1)}\le\cdots\le y_{(N)}$ for their order statistics; throughout the
paper the empirical $p_0$-quantile is $y_{(\lceil p_0N\rceil)}$.
For a compact local search interval $\mathcal A\subseteq[0,\bar\alpha]$
and a candidate $\widetilde\alpha\in\mathcal A$, define
$R_s(\widetilde\alpha)=d_s-\widetilde\alpha x_s$ after a loss and
$R_s(\widetilde\alpha)=-\infty$ after a win.
Algorithm~\ref{alg:quantile} specifies the rank groups, their empirical
quantiles $\widehat q_k(\widetilde\alpha)$, and the balancing objective.
Whether either quantile is $-\infty$ depends only on the group's number of
wins, so the stated fallback applies to all candidates and defines the
estimator for small or heavily censored samples.

\begin{algorithm}[htbp]
\caption{Quantile-Based Estimator}
\label{alg:quantile}
\begin{enumerate}
    \item \textbf{Input:} observed samples
    $\{(x_s,w_s,(1-w_s)d_s)\}_{s=1}^n$, where
    $w_s=\I\{b_s>d_s\}$ is the observed win indicator; a compact local
    interval $\mathcal A$; a fixed reference point $a_{\mathrm{ref}}\in\mathcal A$;
    and $p_0=0.99$.
    \item If $n<2$, output $a_{\mathrm{ref}}$ and return. Otherwise sort the
    observations by $x_s$, breaking ties by time index. Let $I_1$ contain
    the first $\lfloor n/2\rfloor$ observations and $I_2$ the remaining
    $\lceil n/2\rceil$ observations.
    \item If $\sum_{s\in I_k}w_s\ge\lceil p_0|I_k|\rceil$ for either group,
    output $a_{\mathrm{ref}}$ and return.
    \item For each $\widetilde\alpha\in\mathcal A$, construct
    $R_s(\widetilde\alpha)$, compute the two finite quantiles, and set
    $Q(\widetilde\alpha)=|\widehat q_1(\widetilde\alpha)-\widehat q_2(\widetilde\alpha)|$.
    \item Output the smallest minimizer of $Q$ over $\mathcal A$.
\end{enumerate}
\end{algorithm}

\medskip
\noindent
\textbf{Key Property: Quantile Invariance.}
Initial exploration restricts the parameter search to a small interval
$\mathcal A_0$ around a preliminary estimate; its construction is given in
Section~\ref{sec:algorithm}. The initial-localization event means that this
interval contains the true coefficient $\alpha$ and satisfies
$\sup_{\widetilde\alpha\in\mathcal A_0}
|\widetilde\alpha-\alpha|\le\beta_0$, where
$\beta_0=\widetilde O(T^{-1/4})$. Appendix~\ref{app:initial-localization}
shows that this event has high probability. Write
$R_j^o(\widetilde\alpha)=d_j-\widetilde\alpha x_j$ for the uncensored
residual and $\widehat q_k^o(\widetilde\alpha)$ for its empirical
$p_0$-quantile in group $I_k$. The following lemma establishes uniform
quantile invariance over $\mathcal A_0$.

\begin{lemma}[Conditional Quantile Invariance]
\label{lem:quantile-invariance}
Suppose $X$ is atomless. Fix a phase $i$, let $n_i=|A_i|\ge\sqrt T$, and let $I_1,I_2$ be the
rank-based lower- and upper-half groups. On the initial-localization event, suppose $\sup_{\widetilde\alpha\in\mathcal A_0} |\widetilde\alpha-\alpha| \le\beta_0$. Let $c_{\mathrm{vc}}=1024$ be a numerical constant for the uniform
concentration bounds used below, and set $C_q:=2c_{\mathrm{vc}}/g_{\min}$
and $r_{n_i}(\delta):=C_q\sqrt{\log(2n_i/\delta)/n_i}$.
Under Assumptions~\ref{ass:lipschitz}, \ref{ass:bounded}, and
\ref{ass:visible-quantile}, suppose $b_j\le v_j$ for every $j\in A_i$,
$2\bar x\beta_0+4r_{n_i}(\delta)\le\min\{\Delta_q,r_g\}$, and
$c_{\mathrm{vc}}\sqrt{\log(2n_i/\delta)/n_i}\le(1-p_0)/4$.
Then, conditional on the initial exploration history, with probability
at least $1-\delta$,
$\widehat q_k(\widetilde\alpha)=\widehat q_k^o(\widetilde\alpha)$
simultaneously for all $\widetilde\alpha\in\mathcal A_0$ and $k=1,2$.
\end{lemma}

On the stated event, the estimator uses exactly the same quantiles as it
would with uncensored data. The proof combines a deterministic bound on
censored residuals with empirical quantile concentration; see
Appendix~\ref{app:quantile-invariance}.

\medskip
\noindent
\textbf{Confidence Interval.}
The quantile invariance property yields the following finite-sample guarantee (proof in Appendix~\ref{app:main-estimation}).
\begin{theorem}[Estimation Error]
\label{thm:main-estimation}
For any $\delta\in(0,1)$, fix a phase $i$ with $|A_i|\ge \sqrt T$.
Suppose Assumptions~\ref{ass:lipschitz}, \ref{ass:bounded},
\ref{ass:identifiability}, and~\ref{ass:visible-quantile} hold,
Algorithm~\ref{alg:quantile} is run on $A_i$ with search interval
$\mathcal A=\mathcal A_0$, and the conditions of
Lemma~\ref{lem:quantile-invariance} hold. Conditional on an initial
exploration history for which the initial-localization event holds, if
$\beta_0\le r_{\mathrm{id}}$, then with probability at least $1-\delta$,
$|\hat\alpha_i-\alpha|\le C_\alpha\sqrt{\log(2|A_i|/\delta)/|A_i|}$,
where $C_\alpha:=4C_q/\kappa$.
On this event the fallback in Algorithm~\ref{alg:quantile} is not used.
\end{theorem}

\section{Algorithm Design}
\label{sec:algorithm}

We design the algorithm through a Lagrangian relaxation of the budget
constraint. For the analysis, consider the expected-budget relaxation $\mathbb E_{v,d}^{\pi} [ \sum_{t=1}^T \I\{b_t^\pi>d_t\}b_t^\pi ] \le \rho T$. For a multiplier $\lambda\ge0$, the corresponding expected per-round Lagrangian reward is $\bigl(v_t-(1+\lambda)b_t^\pi\bigr) G(b_t^\pi-\alpha x_t) +\lambda\rho$. By weak duality,
\begin{equation}
\label{eq:weak-duality}
\max_{\pi\in\Pi_B}R(\pi)
\le
T\min_{\lambda\ge0}D(\lambda),
\end{equation}
where
\begin{equation}
\label{eq:population-dual}
D(\lambda)
:=
\mathbb E_X
\left[
    \max_{0\le b\le f(X)}
    \bigl(f(X)-(1+\lambda)b\bigr)
    G(b-\alpha X)
\right]
+\lambda\rho.
\end{equation}
Fix
\begin{equation}
\label{eq:dual-optimum}
\lambda^\star
\in
\arg\min_{\lambda\ge0}D(\lambda).
\end{equation}
The algorithm approximately optimizes this relaxed objective while enforcing
the actual budget by stopping once the remaining budget is below $\bar v$.
The resulting loss from early stopping is controlled in the regret analysis.

For a fixed observed context $x$, the exact identity underlying this
relaxation is
\begin{equation}
\label{eq:fixed-context-lagrangian}
r(v,b,x)-\lambda c(b,x)
=(1+\lambda)r\!\left(\frac{v}{1+\lambda},b,x\right).
\end{equation}
Here the first argument of $r$ is a numerical reward coefficient: the
context $x$ and the competing-bid distribution conditional on $x$ remain
fixed. The budget adjustment does not replace $x$ by a context whose value
equals $v/(1+\lambda)$.

Both $G$ and $\alpha$ are unknown, so the algorithm estimates the reward and
cost functions in~\eqref{eq:reward-cost-functions}. The main idea is to work
in the shifted coordinate $\widetilde b=b-\widehat\alpha x$, which removes
the estimated contextual component and allows observations from different
contexts to be compared on a common scale. For a coefficient estimate
$\widehat\alpha$ and a dual multiplier $\lambda$, define the effective value $q=\frac{v}{1+\lambda}-\widehat\alpha x$. This is the value remaining after the budget adjustment and the contextual
shift. Substituting $\widehat\alpha$ for $\alpha$ in the bid-dependent part
of the Lagrangian gives $\bigl(v-(1+\lambda)b\bigr)G(b-\widehat\alpha x) =(1+\lambda)(q-\widetilde b)G(\widetilde b)$. Thus, up to the positive factor $1+\lambda$, the bidding objective becomes
$r_q(q,\widetilde b):=(q-\widetilde b)G(\widetilde b)$. The effective value
$q$ therefore provides a common index for bidding decisions across contexts
and budget multipliers. The algorithm increases the multiplier when
estimated expenditure exceeds the average budget per round $\rho$ and
decreases it toward zero when estimated expenditure is below $\rho$.
A larger multiplier reduces the effective value and discourages larger bids.

\paragraph{Inputs and confidence constants.}
The algorithm takes $T$, $\rho$, $p_0$, the domain bounds
$\bar x,\bar v,\bar z,\bar\alpha$, and valid numerical bounds
$g_{\max},g_{\min},\kappa$ as inputs. It uses
$C_q=2c_{\mathrm{vc}}/g_{\min}$, $C_\alpha=4C_q/\kappa$, and
$C_\omega=64$. Conservative certified values of the three regularity
bounds are allowed; the algorithm and the displayed regret bound use the
same supplied values.
The coefficient $\alpha$ and the distribution $G$ remain unknown.
The neighborhood radii $r_g,r_{\mathrm{id}}$, visibility margin $\Delta_q$,
effective-value margin $\gamma_q$, value-density bounds, and
$\operatorname{Var}(X)$ enter the guarantee and its horizon requirements,
but are not inputs to the bidding rule.

\paragraph{Discretization and phases.}
Set $\delta:=\min\{1/2,T^{-2}\}$ and
\begin{equation}
\label{eq:budget-scaled-stepsize}
\eta:=\min\left\{\frac1{\bar v},\frac1{\rho\sqrt T}\right\}.
\end{equation}
The cap ensures $\eta\bar v\le1$ for every horizon. The second term balances
the dual-update error against the loss from stopping when the budget is
nearly exhausted; the two terms are displayed in Section~\ref{sec:regret-analysis}.
Let $M=K=\lceil\sqrt T\rceil$, $h_q=\bar v/M$, $h_b=\bar v/K$,
and $\bar\lambda=\bar v/\rho$.
Define $Q=\{q^m=mh_q:m=0,1,\ldots,M\}$ and
$\widetilde{\mathcal B}=\{\widetilde b^k=kh_b:k=0,1,\ldots,K\}$.
The grids discretize the effective value and the shifted bid, respectively.
At round $t$ in phase $i$, write
$q_t=v_t/(1+\lambda_t)-\widehat\alpha_{i-1}x_t$ for the current effective
value, where $\lambda_t$ is the current dual multiplier. The algorithm
assigns $q_t$ to the largest grid point not
exceeding its projection onto $[0,\bar v]$. Each bin therefore groups bidding
decisions with similar effective values, even when their contexts differ.
The selected shifted bid is converted back to an actual bid using the
observed context $x_t$. For every effective-value bin $m$, the algorithm
maintains an active shifted-bid set, consisting of the shifted bids still
under consideration for that bin. Binning by $q_t$ is necessary because the
multiplier changes over time: an ordering of the unadjusted values alone
does not imply the same ordering of the bidding objectives. The
effective-value margin supplies the comparison $\gamma_qv_t\le q_t\le v_t$
used later to count informative observations. Each active set satisfies
$\mathcal B_i^m\subseteq\widetilde{\mathcal B}$; initialize
$\mathcal B_0^m=\widetilde{\mathcal B}$. For the analysis, write
\begin{equation}
\label{eq:grid-error}
\varepsilon_{\mathrm{grid}}:=h_q+(1+\bar v g_{\max})h_b.
\end{equation}
The factor multiplying $h_b$ is the Lipschitz bound for the shifted reward
as a function of the bid.

The horizon is divided into an initial exploration block and doubling
phases. Let $H_0=\lceil\sqrt T\rceil$, $n_0=\min\{T,2H_0\}$, and
$T_0=\{1,\ldots,n_0\}$. For $i\ge1$, set
$\ell_i=2^{i-1}H_0$ and $a_i=n_0+2H_0(2^{i-1}-1)+1$.
Define $A_i=\{a_i,\ldots,a_i+\ell_i-1\}\cap\{1,\ldots,T\}$ and
$U_i=\{a_i+\ell_i,\ldots,a_i+2\ell_i-1\}\cap\{1,\ldots,T\}$.
Let $n$ be the number of nonempty phases, with $n=0$ if $n_0=T$.
Every phase $i<n$ has two full blocks. In the final phase, either block may
be truncated or empty. The estimation block $A_i$ estimates $\alpha$ and
the update block $U_i$ updates the distribution and active bid sets.
Throughout $A_i\cup U_i$, bidding and shifted observations use
$\widehat\alpha_{i-1}$, so the coordinate system is fixed within a phase.
For $i<n$, the estimate computed at the end of $A_i$ is first used in phase
$i+1$. No estimator or active-set update is computed from the final phase,
since it would not be used for a subsequent decision. Thus every block
used to construct an estimate has at least $H_0$ observations.

\paragraph{Initial exploration.}
During $T_0$, the algorithm bids $b_t=0$, so all competing bids $d_t$ are
observed. These uncensored samples provide the initial ordinary
least-squares (OLS) estimate
\begin{equation}
\label{eq:alpha-estimate}
\widehat\alpha_0
=
\Pi_{[0,\bar\alpha]}
\left[
\frac{
    \sum_{s\in T_0}
    (x_s-\bar x_{T_0})(d_s-\bar d_{T_0})
}{
    \sum_{s\in T_0}
    (x_s-\bar x_{T_0})^2
}
\right].
\end{equation}
Here $\bar x_{T_0}$ and $\bar d_{T_0}$ are the sample averages.
If the denominator in~\eqref{eq:alpha-estimate} is zero, set
$\widehat\alpha_0=0$.
Write $S_{xx}:=\sum_{s\in T_0}(x_s-\bar x_{T_0})^2$ and set
\begin{equation}
\label{eq:initial-localization-radius}
r_0=
\begin{cases}
\displaystyle\min\left\{\bar\alpha,
\bar z\sqrt{\frac{\log(8/\delta)}{2S_{xx}}}\right\},&S_{xx}>0,\\[2pt]
\bar\alpha,&S_{xx}=0,
\end{cases}
\qquad \beta_0=2r_0.
\end{equation}
This radius is computed from the observed design, without knowing
$\operatorname{Var}(X)$. Initialize
\begin{equation}
\label{eq:initial-search-interval}
\mathcal A_0
=[\widehat\alpha_0-r_0,\widehat\alpha_0+r_0]\cap[0,\bar\alpha].
\end{equation}
The localization guarantee in Appendix~\ref{app:initial-localization}
ensures that this interval contains $\alpha$ and has radius of order
$\widetilde O(T^{-1/4})$ with high probability. In later phases,
Algorithm~\ref{alg:quantile} searches over this interval using the samples
in $A_i$, with reference point $a_{\mathrm{ref}}=\widehat\alpha_0$. We also initialize the noise-distribution estimator by
\begin{equation}
\label{eq:initial-G-estimator}
\widehat G_0(z)
=
\frac1{|T_0|}
\sum_{s\in T_0}
\I\{d_s-\widehat\alpha_0x_s\le z\}.
\end{equation}
\paragraph{Shifted-coordinate estimators.}
For every $s\in A_i\cup U_i$, define $\widetilde b_s = b_s-\widehat\alpha_{i-1}x_s$. On a losing round, when $d_s$ is observed, also define $\widetilde d_s = d_s-\widehat\alpha_{i-1}x_s$. For every $\widetilde b^k\in\widetilde{\mathcal B}$, define the observable
indicator $Y_s^k=1$ after a win and $Y_s^k=\I\{\widetilde d_s\le\widetilde b^k\}$ after a loss. An observation contributes at threshold $\widetilde b^k$ whenever
$\widetilde b_s\le\widetilde b^k$. On a winning round, this implies
$d_s-\widehat\alpha_{i-1}x_s<\widetilde b_s\le\widetilde b^k$, so the
relevant indicator equals one. On a losing round, $d_s$ is observed and the
indicator can be evaluated directly. Hence the following estimator is
implementable under one-sided feedback:
\begin{equation}
\label{eq:G-estimator}
n_i^k
=
\sum_{s\in U_i}
\I\{\widetilde b_s\le\widetilde b^k\},
\qquad
\widehat G_i(\widetilde b^k)
=
\frac{
    \sum_{s\in U_i}
    \I\{\widetilde b_s\le\widetilde b^k\}Y_s^k
}{
    n_i^k\vee1
}.
\end{equation}
For $i\ge0$, using~\eqref{eq:initial-G-estimator} when $i=0$, define
\begin{equation}
\label{eq:reward-cost-estimators}
\widehat r_i(q^m,\widetilde b^k)
=
(q^m-\widetilde b^k)\widehat G_i(\widetilde b^k),
\qquad
\widehat c_i(\widetilde b^k;x)
=
(\widetilde b^k+\widehat\alpha_i x)
\widehat G_i(\widetilde b^k).
\end{equation}
The cost estimator uses $\widehat\alpha_i$ because it will be evaluated
in the next phase. It is used only at feasible bids, namely when
$0\le\widetilde b^k+\widehat\alpha_i x\le f(x)$.

For each update index $1\le i<n$, set $\delta_i:=\delta/[2(n\vee1)]$, $\beta_i:=\min\{\beta_0,C_\alpha\sqrt{\log(2|A_i|/\delta_i)/|A_i|}\}$, and $L_T:=\log(8T^2(K+1)(M+1)/\delta)$. The first term in $\beta_i$ uses the fact that all estimates stay in
$\mathcal A_0$; the second is supplied by Theorem~\ref{thm:main-estimation}.
For every update index $1\le i<n$ and bin $m$, define
\begin{equation}
\label{eq:confidence-radius}
N_i^m
=
\min_{\widetilde b^k\in\mathcal B_{i-1}^m}
n_i^k,
\qquad
\omega_i^m
=
\min
\left\{
2\bar v,\,
C_\omega
\left[
    \bar v
    \sqrt{
        \frac{L_T}{N_i^m\vee1}
    }
    +
    (1+\bar v g_{\max})
    \bar x(\beta_{i-1}+\beta_i)
\right]
\right\}.
\end{equation}
The first term accounts for sampling error in $\widehat G_i$. The second
accounts for both contextual shifts: $\widehat\alpha_{i-1}$ is used to
construct the residual observations, whereas $\widehat\alpha_i$ is used to
convert a shifted bid into an actual bid in the next phase. Including both
radii avoids requiring a deterministic ratio between consecutive estimation
errors. The truncation at $2\bar v$ uses the
boundedness of the bidding objective $r_q$ and its estimate, and of the
costs at feasible bids.

\paragraph{Action selection and filtering.}
The active-set update uses the following ordering property of the oracle
shifted bids.

\begin{lemma}[Monotonicity of oracle shifted bids]
\label{lem:residual-monotonicity}
For every $m=0,\ldots,M$, define
$\widetilde b_m^\star := \min\arg\max_{\widetilde b^k\in\widetilde{\mathcal B}} r_q(q^m,\widetilde b^k)$,
where $r_q$ is the shifted reward defined above.
Then $\widetilde b_m^\star\le q^m$, and $m_1\le m_2$ implies
$\widetilde b_{m_1}^\star\le\widetilde b_{m_2}^\star$.
\end{lemma}

\begin{proof}{Proof.}
The inequality $\widetilde b_m^\star\le q^m$ follows immediately from
$0\in\widetilde{\mathcal B}$ and the definition of the smallest maximizer.

Fix $m_1<m_2$ and write $q_j=q^{m_j}$ and
$\widetilde b_j=\widetilde b_{m_j}^\star$, $j=1,2$. By optimality,
$r_q(q_1,\widetilde b_1)\ge r_q(q_1,\widetilde b_2)$ and
$r_q(q_2,\widetilde b_2)\ge r_q(q_2,\widetilde b_1)$.
Adding these inequalities gives $(q_2-q_1)\bigl(G(\widetilde b_2)-G(\widetilde b_1)\bigr)\ge0$. Suppose that $\widetilde b_1>\widetilde b_2$. Since $G$ is nondecreasing,
the preceding inequality implies $G(\widetilde b_1)=G(\widetilde b_2)$. If this
common value is positive, then $r_q(q_1,\widetilde b_2)-r_q(q_1,\widetilde b_1) =(\widetilde b_1-\widetilde b_2)G(\widetilde b_1)>0$, contradicting the optimality of $\widetilde b_1$. If it is zero, then
$\widetilde b_2$ is also a maximizer at $q_1$, contradicting the definition
of $\widetilde b_1$ as the smallest maximizer. Hence $\widetilde b_1\le\widetilde b_2$. \hfill\Halmos
\end{proof}

This ordering motivates the
cross-bin filtering rule. After eliminating bids whose estimated rewards
are more than $2\omega_i^m$ below the estimated best reward, the algorithm
enforces the same order on the smallest remaining bids across bins.
Selecting the smallest feasible active bid promotes exploration, since the
resulting observation can be used at every larger shifted-bid threshold.
Together, these steps shrink the active sets while preserving an optimal
shifted bid in each bin when the stated estimation error bounds hold
simultaneously.

As in~\citet{wang2023learning}, the algorithm combines a dual gradient
update with active-set bid selection. The additional component is learning
the contextual parameter from one-sided feedback and incorporating its
estimation error into the confidence bounds.

\begin{algorithm}[htbp]
\caption{Bidding Algorithm for Contextual First-Price Auctions with Budgets}
\label{alg:budget-bidding}
\begin{enumerate}
    \item \textbf{Input and initialization.} Take $T,\rho,p_0$, the domain
    bounds, and $g_{\max},g_{\min},\kappa$ as specified above. Set $B=\rho T$.
    If $B<\bar v$, bid zero throughout and return. Use the grids, confidence
    constants, and truncated phase schedule defined above, with
    $\eta=\min\{\bar v^{-1},(\rho\sqrt T)^{-1}\}$ and
    $\delta=\min\{1/2,T^{-2}\}$. Let $B_t$ denote the remaining budget
    before round $t$, and initialize $B_1=B$ and
    $\mathcal B_0^m=\widetilde{\mathcal B}$ for all $m$.
    \item \textbf{Initial exploration.} For $t\in T_0$, bid zero, observe
    $d_t$, keep $B_{t+1}=B$, and set $\lambda_t=\widehat c_t=0$.
    If $|T_0|=T$, return. Otherwise compute $\widehat\alpha_0,r_0,\beta_0$,
    $\mathcal A_0,\widehat G_0$ by
    \eqref{eq:alpha-estimate}--\eqref{eq:initial-G-estimator}, and set
    $\lambda_{|T_0|+1}=0$.
    \item \textbf{For phase $i=1,\ldots,n$:} keep
    $\widehat\alpha_{i-1},\widehat G_{i-1}$ and $\{\mathcal B_{i-1}^m\}$
    fixed throughout $A_i\cup U_i$.
    \begin{enumerate}
        \item For each $t\in A_i$, compute
        $q_t=v_t/(1+\lambda_t)-\widehat\alpha_{i-1}x_t$ and
        $m(t)=\max\{m:q^m\le\Pi_{[0,\bar v]}(q_t)\}$.
        Form the feasible active set $\mathcal F_t$ from the bids
        $\widetilde b\in\mathcal B_{i-1}^{m(t)}$ satisfying
        $0\le\widetilde b+\widehat\alpha_{i-1}x_t\le v_t/(1+\lambda_t)$.
        If it is nonempty, set $\widetilde b_t=\min\mathcal F_t$,
        $b_t=\widetilde b_t+\widehat\alpha_{i-1}x_t$, and
        $\widehat c_t=\widehat c_{i-1}(\widetilde b_t;x_t)$.
        Otherwise set $b_t=\widehat c_t=0$. Submit the bid, observe the win
        indicator and $d_t$ after a loss, and update
        $\lambda_{t+1}=\max\{0,\lambda_t+\eta(\widehat c_t-\rho)\}$ and
        $B_{t+1}=B_t-b_t\I\{b_t>d_t\}$.
        If $B_{t+1}<\bar v$, set all subsequent bids to zero and return.
        \item If $i<n$, at the end of $A_i$ compute $\widehat\alpha_i$ with
        Algorithm~\ref{alg:quantile}, using $\mathcal A=\mathcal A_0$ and
        $a_{\mathrm{ref}}=\widehat\alpha_0$.
        \item For each $t\in U_i$, repeat part~(a), still using
        $\widehat\alpha_{i-1}$.
        \item If $i<n$, at the end of $U_i$ compute the estimators and
        radii in~\eqref{eq:G-estimator}--\eqref{eq:confidence-radius}.
        For each bin, eliminate candidates by forming $\widehat{\mathcal B}_i^m$ from the bids in $\mathcal B_{i-1}^m$ whose estimated rewards are within $2\omega_i^m$ of $\max_{\widetilde b\in\mathcal B_{i-1}^m}\widehat r_i(q^m,\widetilde b)$. Set $\mathcal B_i^0=\widehat{\mathcal B}_i^0$. For
        $m=1,\ldots,M$ in increasing order, set $\mathcal B_i^m$ to the subset of $\widehat{\mathcal B}_i^m$ containing bids no smaller than $\max_{\ell<m}\min\mathcal B_i^\ell$. If this set is empty, replace it by the singleton
        $\{\max_{\ell<m}\min\mathcal B_i^\ell\}$.
        No updates are computed from the final phase $i=n$.
    \end{enumerate}
\end{enumerate}
\end{algorithm}

\section{Regret Analysis}
\label{sec:regret-analysis}

We compare Algorithm~\ref{alg:budget-bidding} with the optimal hard-budget
policy by first passing to the expected-budget relaxation. The proof then
has four parts: simultaneous estimation bounds preserve an oracle bid in
every bin; the effective-value ordering supplies informative observations;
these observations control the sum of confidence radii; and the dual update
controls expenditure and the loss from early stopping.

To state the dependence on the problem parameters, write
\begin{equation}
\label{eq:scalar-statistical-constants}
\sigma_X^2:=\operatorname{Var}(X),
\qquad \sigma_X:=\sqrt{\sigma_X^2},
\qquad C_{\mathrm{init}}:=\frac{2\bar z}{\sigma_X},
\qquad \kappa_V:=\frac{\underline h_V\gamma_q}{\overline h_V},
\end{equation}
and recall $C_q=2c_{\mathrm{vc}}/g_{\min}$ and $C_\alpha=4C_q/\kappa$.
Here $\sigma_X^2>0$ follows from the atomless context distribution.
$C_{\mathrm{init}}$ controls initial localization, $C_\alpha$ controls
quantile-based estimation, and $\kappa_V$ measures the probability of
observing a sufficiently small value. Define
\begin{equation}
\label{eq:scalar-statistical-factor}
\mathcal C_{\mathrm{stat}}
:=
\bar v\left(1+\bar v g_{\max}+\kappa_V^{-1/2}\right)
+(1+\bar v g_{\max})\bar x(C_{\mathrm{init}}+C_\alpha).
\end{equation}
These quantities are used to state the bound; the algorithm does not need
$\sigma_X$ or $\kappa_V$.

For the analysis, let $\tau$ be the last round before the budget safeguard
stops bidding. Continue the algorithm formally after $\tau$ on the same
independent sequence of contexts and noise, following the deterministic
phase schedule and all estimation, selection, and dual-update rules, but
without the budget stopping rule. The actual policy and this continuation
agree through $\tau$. The ordering and confidence statements below
refer to this continuation.

Before proving the regret bound, we record the ordering property of the
bids selected by this continuation. Here the common confidence event
means that the parameter bounds and the distribution bounds of
Lemma~\ref{lem:confidence-representation} hold simultaneously; that lemma
also retains the oracle candidates of
Lemma~\ref{lem:residual-monotonicity}.
The probability of this event is established in the proof of
Theorem~\ref{thm:main-regret} below.

\begin{lemma}[Monotonicity of selected shifted bids]
\label{lem:selected-shifted-bid-monotonicity}
Fix a phase \(i\) and work on the common good event. If
\(s,t\in A_i\cup U_i\) satisfy \(q_s\le q_t\), then $\widetilde b_s\le\widetilde b_t$. Moreover, if \(i\ge2\), \(s\in U_{i-1}\),
\(t\in A_i\cup U_i\), and \(q_s\le q_t\), then
$\widetilde b_s\le\widetilde b$ for every
$\widetilde b\in\mathcal B_{i-2}^{m(t)}$.
\end{lemma}

\begin{proof}{Proof.}
Since the projection and binning maps are nondecreasing,
\(q_s\le q_t\) implies \(m(s)\le m(t)\). The monotone closure gives
$\inf\mathcal B_j^{m_1} \le \inf\mathcal B_j^{m_2}, \qquad m_1\le m_2.$
The retained oracle satisfies $\widetilde b_{m(s)}^\star \le q^{m(s)} \le q_s$, so the smallest active shifted bid is feasible. Hence, for
\(s,t\in A_i\cup U_i\),
$\widetilde b_s = \inf\mathcal B_{i-1}^{m(s)} \le \inf\mathcal B_{i-1}^{m(t)} = \widetilde b_t.$
For \(s\in U_{i-1}\), the bid at \(s\) uses
\(\mathcal B_{i-2}^{m(s)}\), giving
$\widetilde b_s = \inf\mathcal B_{i-2}^{m(s)} \le \inf\mathcal B_{i-2}^{m(t)} \le \widetilde b$
for every
\(\widetilde b\in\mathcal B_{i-2}^{m(t)}\).
\hfill\Halmos
\end{proof}

\begin{theorem}[Main Result]
\label{thm:main-regret}
Under Assumptions~\ref{ass:lipschitz}, \ref{ass:bounded},
\ref{ass:identifiability}, \ref{ass:visible-quantile},
\ref{ass:effective-value-margin}, and~\ref{ass:value-density},
Algorithm~\ref{alg:budget-bidding}, with the step size
in~\eqref{eq:budget-scaled-stepsize}, satisfies $\operatorname{Regret}(\pi) \le C(1+\frac{\bar v}{\rho}) \mathcal C_{\mathrm{stat}}\sqrt T\log T$ for every $T\ge T_\star$, where $C$ is a universal numerical constant.
An explicit set of sufficient conditions on $T_\star$, including all
localization and quantile-neighborhood requirements, is given
in~\eqref{eq:explicit-large-horizon}. For fixed problem parameters, the
bound is $\widetilde O(\sqrt T)$.
\end{theorem}

The statistical factor separates the direct budget dependence from the
statistical constants. The effective-value margin $\gamma_q$, and hence
$\kappa_V$ and $T_\star$, may themselves depend on $\rho$; the theorem does
not assert a uniform $1/\rho$ dependence when these quantities vary.
\begin{proof}{Proof of Theorem~\ref{thm:main-regret}.}
\label{app:proof-main-regret}
Fix $\delta=T^{-2}$ and assume the sufficient horizon conditions
in~\eqref{eq:explicit-large-horizon}; the corresponding quantities
$\ell_T$ and $\overline\beta_0(T)$ are defined in
Appendix~\ref{app:main-regret-horizon}.

\paragraph{Estimation and oracle retention.}
For each update $i<n$, $|A_i|\ge H_0$, $n\le T$, and
$\delta_i=\delta/[2(n\vee1)]$, so
$\log(2|A_i|/\delta_i)\le\ell_T$.
Let $\mathcal E_0$ be the event in Lemma~\ref{lem:initial-localization},
including the lower bound on $S_{xx}$. It has probability at least
$1-\delta/2$, and on it $\beta_0\le\overline\beta_0(T)$.
Conditional on any initial history in this event, the quantile and
visibility conditions hold for every block used in an update.
Theorem~\ref{thm:main-estimation} and a union bound therefore give
$\mathcal E_\alpha:=\mathcal E_0\cap
\bigcap_{1\le i<n}\{|\widehat\alpha_i-\alpha|\le\beta_i\}$ with
$\mathbb P(\mathcal E_\alpha^c)\le\delta$.
Let $\mathcal E_G$ be the martingale confidence event of
Lemma~\ref{lem:confidence-representation}, and set
$\mathcal E=\mathcal E_\alpha\cap\mathcal E_G$.
This is the event on which initial localization, all subsequent parameter
bounds, and all distribution confidence bounds hold, and
$\mathbb P(\mathcal E^c)\le2\delta$.
Lemma~\ref{lem:confidence-representation}
first bounds reward and cost errors using $\omega_i^m$. It then shows that
confidence elimination and the sequential cross-bin filtering retain $\widetilde b_m^\star :=\min\arg\max_{\widetilde b\in\widetilde{\mathcal B}} (q^m-\widetilde b)G(\widetilde b)$. This oracle candidate is defined using the true distribution and is used
only in the analysis. Since $G(0)=0$ and $G$ is nondecreasing, this smallest maximizing grid bid
satisfies $\widetilde b_m^\star\le q^m$. Thus, at the current bin,
$\widetilde b_{m(t)}^\star\le q^{m(t)}\le q_t$ is feasible. In particular,
the feasible active set is nonempty on $\mathcal E$, and the algorithm
selects its smallest active bid. Establishing this fact before counting
samples is essential: it makes the cross-bin ordering apply to the bids
actually selected by the algorithm.

\paragraph{Informative observations and the confidence sum.}
By Lemma~\ref{lem:dual-bound} and the effective-value margin,
$\gamma_qv_t\le q_t\le v_t$. If $s\in U_{i-1}$ and
$t\in A_i\cup U_i$, then $v_s\le\gamma_qv_t \quad\Longrightarrow\quad q_s\le v_s\le\gamma_qv_t\le q_t$. The monotone filtering rule consequently makes the observation at $s$
usable for every candidate whose count enters $N_{i-1}^{m(t)}$; see
Lemmas~\ref{lem:selected-shifted-bid-monotonicity} and~\ref{lem:rank-lower-bound}.
This argument allows both the multiplier and the parameter estimate to
differ between the two rounds. The weaker condition $v_s\le v_t$ alone
would not provide the required comparison.

Writing $F_V$ for the value distribution function, the density bounds give
$F_V(\gamma_qv)\ge\kappa_V F_V(v)$. The number of such observations can
therefore be bounded using the i.i.d. values, without conditioning their
distribution on $\mathcal E$. Summing over the doubling phases yields $\mathbb E\![ \sum_{i=2}^n\sum_{t\in A_i\cup U_i}\omega_{i-1}^{m(t)} ] \le C\mathcal C_{\mathrm{stat}}\sqrt T\log T$. The sampling contribution contains $\kappa_V^{-1/2}$, while the contextual
estimation contribution contains $C_{\mathrm{init}}+C_\alpha$.
This is Lemma~\ref{lem:confidence-sum-bound}.

\paragraph{Relaxed benchmark and regret decomposition.}
Let $R_{\mathrm{rel}}^\star:=TD(\lambda^\star)$.
Lemma~\ref{lem:relaxed-stationary-benchmark} provides a stationary bidding
rule with expected expenditure at most $\rho$ and expected total reward
$R_{\mathrm{rel}}^\star$. If $b_t^\star$ is an independent draw from this
rule conditional on the current context and value, then $R(\pi^\star)\le R_{\mathrm{rel}}^\star =\mathbb E\![\sum_{t=1}^T r(v_t,b_t^\star,x_t)]$. Let $I_0:=T_0\cup A_1\cup U_1$ and use the formal continuation
defined above. Since each omitted reward is at most $\bar v$ and
$|I_0|=O(\sqrt T)$,
\begin{align}
\operatorname{Regret}(\pi)
&\le\bar v\,\mathbb E[T-\tau]+C\bar v\sqrt T
\nonumber\\
&\quad+\mathbb E\!\left[
\I\{\mathcal E\}\sum_{i=2}^n\sum_{t\in A_i\cup U_i}
\bigl(r(v_t,b_t^\star,x_t)-r(v_t,b_t,x_t)\bigr)
\right]+C\bar vT\mathbb P(\mathcal E^c).
\label{eq:regret-decomposition}
\end{align}
The confidence and oracle-retention argument gives, on $\mathcal E$,
\begin{equation}
r(v_t,b_t^\star,x_t)-r(v_t,b_t,x_t)
\le C(1+\bar\lambda)[\omega_{i-1}^{m(t)}+\varepsilon_{\mathrm{grid}}]
+\lambda_t[c(b_t^\star,x_t)-\widehat c_t],
\label{eq:one-step-lagrangian-bound}
\end{equation}
where $\varepsilon_{\mathrm{grid}}$ is defined in~\eqref{eq:grid-error}.
This comparison uses the fixed-context identity
in~\eqref{eq:fixed-context-lagrangian}. Its accumulated discretization
error satisfies $T\varepsilon_{\mathrm{grid}}
\le C\bar v(1+\bar v g_{\max})\sqrt T
\le C\mathcal C_{\mathrm{stat}}\sqrt T$.

\paragraph{Budget control and choice of step size.}
Lemma~\ref{lem:c-bound} bounds the full, unconditional dual sum by
$C\eta T\bar v^2$. To restrict this sum to $\mathcal E$ and the post-initial
phases, we also account for the omitted initial terms and the failure event.
Each summand is bounded in absolute value by $\bar\lambda\bar v$, because
$0\le\lambda_t\le\bar\lambda$ and both costs belong to $[0,\bar v]$.
Since $|I_0|=O(\sqrt T)$ and $\mathbb P(\mathcal E^c)\le2T^{-2}$,
$\mathbb E\![\I\{\mathcal E\}\sum_{i=2}^n\sum_{t\in A_i\cup U_i}
\lambda_t(c(b_t^\star,x_t)-\widehat c_t)]
\le C\eta T\bar v^2+C\bar\lambda\bar v\sqrt T
+C\bar\lambda\bar v/T$.
No independence conditional on $\mathcal E$ is assumed.

Lemma~\ref{lem:stopping-time-bound} bounds the early-stopping contribution by
$\bar v\,\mathbb E[T-\tau]
\le C\frac{\bar v}{\rho}\mathcal C_{\mathrm{stat}}\sqrt T\log T
+\frac{\bar v\bar\lambda}{\rho\eta}+C\frac{\bar v}{T}$.
Combining this estimate, the confidence sum, the discretization error, and
the dual bound in~\eqref{eq:regret-decomposition} gives
$\operatorname{Regret}(\pi)
\le C(1+\bar\lambda)\mathcal C_{\mathrm{stat}}\sqrt T\log T
+C\eta T\bar v^2+\bar v\bar\lambda/(\rho\eta)$.
For $T\ge(\bar v/\rho)^2$, our step size is $\eta=(\rho\sqrt T)^{-1}$,
so both $\eta T\bar v^2$ and $\bar v\bar\lambda/(\rho\eta)$ equal
$(\bar v^2/\rho)\sqrt T$.
Since $\mathcal C_{\mathrm{stat}}\ge\bar v$, these terms are absorbed by
$C(1+\bar\lambda)\mathcal C_{\mathrm{stat}}\sqrt T\log T$.
The bounded initial dual terms and the failure-event contribution are also
absorbed by $C(1+\bar v/\rho)\mathcal C_{\mathrm{stat}}\sqrt T\log T$.
Combining these bounds proves Theorem~\ref{thm:main-regret}.
\hfill\Halmos
\end{proof}

\section{Structured Multi-dimensional Extension}
\label{sec:multi-extension}

We briefly discuss a structured multi-dimensional extension of the
contextual bidding model. Let the highest competing bid satisfy
$D_t=\langle\boldsymbol\alpha,\mathbf X_t\rangle+Z_t$, with
$\boldsymbol\alpha\in\Theta\subset\mathbb R_+^d$, where $d$ is the
context dimension and $\Theta$ is a known compact convex parameter domain.
Here $\langle\cdot,\cdot\rangle$ denotes the Euclidean inner product, and
$Z_t\in[0,\bar z]$ is i.i.d. noise, independent of the contexts, drawn
from the unknown distribution $G$. In this section, $\bar x$ bounds
$\|\mathbf X_t\|_2$; the remaining bounds retain their scalar meanings.
We consider contexts on a monotone one-dimensional manifold
$\mathbf X_t=h(S_t)$, $S_t\in[0,1]$, where
$h:[0,1]\to\mathbb R_+^d$ is the manifold map, with each coordinate strictly
increasing. The value map $\phi(s):=f(h(s))$ is also strictly increasing.
The quantile estimator sorts observations by their observed values $v_t$.
Under the manifold assumption, this value order agrees with the latent
order because $\phi$ is strictly increasing; this agreement is used in
the theoretical analysis. Competing bids may still depend on multiple
context features.

To estimate $\boldsymbol\alpha$, we sort the observations by value and
balance residual quantiles across $L\ge d+1$ consecutive bins of nearly
equal size.
For a candidate parameter $\widetilde{\boldsymbol\alpha}$, the residual is
$d_t-\langle\widetilde{\boldsymbol\alpha},\mathbf x_t\rangle$ after a loss
and $-\infty$ after a win. The estimator minimizes the empirical standard
deviation of the $L$ binwise $p_0$-quantiles over the initial local search
set. At the true parameter, the uncensored residual equals $Z_t$, whose
distribution is invariant across the bins. Under the visibility condition,
the same quantile-invariance argument as in the scalar case applies.
The local identifiability condition in
Assumption~\ref{ass:multi-identifiability} requires that, near the true
parameter, a small spread of the population residual quantiles forces the
candidate parameter to be close to the true one. It therefore converts the
quantile error into an error bound for the vector parameter. The initial estimate is the
vector OLS analogue of~\eqref{eq:alpha-estimate}; nondegeneracy of
$\operatorname{Cov}(h(S))$ supplies the initial localization guarantee.
The initial radius is computed from the smallest eigenvalue of the observed
centered Gram matrix. If the matrix is singular, the initial search set is
all of $\Theta$. We also write $\mathcal A_0$ for the vector local search
set defined by these initial localization quantities. Empty bins or nonfinite empirical quantiles are handled by
returning the initial estimate, as specified in
Algorithm~\ref{alg:multi-quantile}. These fallbacks do not occur on the
confidence event under the stated horizon conditions.

The bidding algorithm uses the same effective-value discretization as in
Section~\ref{sec:algorithm}. For the estimate
$\widehat{\boldsymbol\alpha}_{i-1}$ fixed throughout phase $i$, set $q_t=\frac{v_t}{1+\lambda_t} -\langle\widehat{\boldsymbol\alpha}_{i-1},\mathbf x_t\rangle$. With the shifted bid
$\widetilde b=b-\langle\widehat{\boldsymbol\alpha}_{i-1},\mathbf x_t\rangle$,
the bidding objective in shifted coordinates remains
$r_q(q,\widetilde b)=(q-\widetilde b)G(\widetilde b)$.
The smallest optimal shifted bid is therefore nondecreasing in $q$ by
Lemma~\ref{lem:residual-monotonicity}, so the same active-set construction
applies. Moreover, Assumption~\ref{ass:multi-effective-value-margin} gives
$\gamma_qv_t\le q_t\le v_t$. Hence $v_s\le\gamma_qv_t$ implies
$q_s\le q_t$, providing the same sample-counting argument as in the scalar
analysis. Replacing the scalar estimator and contextual shift in
Algorithm~\ref{alg:budget-bidding} by their vector counterparts gives the
structured multi-dimensional algorithm. Algorithms~\ref{alg:multi-quantile}
and~\ref{alg:multi-budget-bidding} specify its estimation and bidding rules.
The initial localization quantities and the vector confidence estimates are
defined in~\eqref{eq:multi-initial-estimate}--\eqref{eq:multi-initial-radius}
and~\eqref{eq:multi-G-estimator}--\eqref{eq:multi-confidence-radius},
respectively; their analysis is in
Appendix~\ref{app:multi-dimensional-guarantee}.

\begin{algorithm}[htbp]
\caption{Multi-dimensional Quantile-Balancing Estimator}
\label{alg:multi-quantile}
\begin{enumerate}
    \item \textbf{Input:} the observed tuples
    $(\mathbf x_t,v_t,w_t,(1-w_t)d_t)$ for $t\in A_i$, with
    $w_t=\I\{b_t>d_t\}$; $\mathcal A_0$; the reference point $\widehat{\boldsymbol\alpha}_0$; $p_0$; and $L$.     If $|A_i|<L$, return $\widehat{\boldsymbol\alpha}_0$.
    \item Sort the observations in $A_i$ by $v_t$, breaking ties by their
    time indices, and divide them into $L\ge d+1$ consecutive empirical bins
    $\mathcal I_{i,1},\ldots,\mathcal I_{i,L}$ of nearly equal size.
    If the number of wins in any bin is at least
    $\lceil p_0|\mathcal I_{i,\ell}|\rceil$, return
    $\widehat{\boldsymbol\alpha}_0$; the corresponding quantile would be
    $-\infty$ for every candidate.
    \item For each $\boldsymbol\theta\in\mathcal A_0$, use the residual
    $R_t(\boldsymbol\theta)=d_t-\langle\boldsymbol\theta,\mathbf x_t\rangle$
    after a loss and $R_t(\boldsymbol\theta)=-\infty$ after a win.
    Let $\widehat q_{i,\ell}(\boldsymbol\theta)$ be its empirical
    $p_0$-quantile in $\mathcal I_{i,\ell}$, and define
    $\widehat Q_i(\boldsymbol\theta) = \left\{ \frac1L\sum_{\ell=1}^L \left[ \widehat q_{i,\ell}(\boldsymbol\theta) -\frac1L\sum_{j=1}^L\widehat q_{i,j}(\boldsymbol\theta) \right]^2 \right\}^{1/2}.$
    \item Output
    $\widehat{\boldsymbol\alpha}_i
    \in\arg\min_{\boldsymbol\theta\in\mathcal A_0}
    \widehat Q_i(\boldsymbol\theta)$, breaking ties lexicographically.
\end{enumerate}
\end{algorithm}

\begin{algorithm}[htbp]
\caption{Structured Multi-dimensional $q$-based Bidding Algorithm}
\label{alg:multi-budget-bidding}
\begin{enumerate}
\item Run Algorithm~\ref{alg:budget-bidding} with the known vector domain
$\Theta$, dimension $d$, and $L\ge d+1$ bins, using the same constants,
grids, phase schedule, and stopping rules.
\item Use~\eqref{eq:multi-initial-estimate}--\eqref{eq:multi-initial-radius}
for initial localization, the vector analogue of~\eqref{eq:initial-G-estimator}, and
Algorithm~\ref{alg:multi-quantile} for each used $A_i$ update.
\item In phase $i$, keep $\widehat{\boldsymbol\alpha}_{i-1}$ fixed and set
$q_t=v_t/(1+\lambda_t)-
\langle\widehat{\boldsymbol\alpha}_{i-1},\mathbf x_t\rangle$.
Select the smallest feasible active shifted bid, convert it by the vector
shift, and update the dual variable and remaining budget as in the scalar
algorithm. At the end of each $U_i$ with $i<n$, use
\eqref{eq:multi-G-estimator}--\eqref{eq:multi-confidence-radius} for
confidence elimination and the same monotone closure, including the
empty-set reset. The new coefficient is first used in phase $i+1$;
the final phase produces no estimates or active-set updates.
\end{enumerate}
\end{algorithm}

For this extension, keep $\kappa_V=\underline h_V\gamma_q/\overline h_V$.
Let $\kappa_h>0$ satisfy $\operatorname{Cov}(h(S))\succeq\kappa_h I_d$,
where $I_d$ is the identity matrix; let $\kappa$ denote the vector local
identifiability constant of Assumption~\ref{ass:multi-identifiability}. Define
\begin{equation}
\label{eq:vector-statistical-factor}
\begin{aligned}
C_{\mathrm{init},d}&:=4\bar z\sqrt{d/\kappa_h},
& C_{\alpha,d}&:=\frac{4C_qL\sqrt d}{\kappa},\\
\mathcal C_{\mathrm{stat},d}
&:=\bar v(1+\bar v g_{\max}+\kappa_V^{-1/2})
+(1+\bar v g_{\max})\bar x(C_{\mathrm{init},d}+C_{\alpha,d}).
\end{aligned}
\end{equation}
Here $\kappa_h$ is the covariance lower bound in
Assumption~\ref{ass:multi-context}, and $\kappa$ is the vector
identifiability constant in Assumption~\ref{ass:multi-identifiability}.

\begin{proposition}[Structured multi-dimensional guarantee]
\label{prop:multi-dimensional-guarantee}
Suppose Assumptions~\ref{ass:lipschitz}, \ref{ass:value-density},
\ref{ass:multi-context}, \ref{ass:multi-effective-value-margin},
\ref{ass:multi-visible-quantile}, and~\ref{ass:multi-identifiability} hold.
Fix a finite dimension $d$ and $L\ge d+1$. With the step size
in~\eqref{eq:budget-scaled-stepsize}, Algorithm~\ref{alg:multi-budget-bidding}
satisfies $\operatorname{Regret}(\pi) \le C(1+\frac{\bar v}{\rho}) \mathcal C_{\mathrm{stat},d}\sqrt T\log T$ for every $T\ge T_\star(d)$ satisfying~\eqref{eq:multi-large-horizon}, where
$C$ is a universal numerical constant. If $L=O(d)$, the displayed
estimation factor is of order $d^{3/2}$ when the other displayed regularity
parameters are held fixed, giving $\widetilde O(d^{3/2}\sqrt T)$ regret.
The dependence on $\kappa_h$ and $\kappa$ is retained explicitly in the bound.
\end{proposition}

The proof follows the same relaxed-benchmark and active-set decomposition as
Theorem~\ref{thm:main-regret}. The additional ingredients are initial
localization and quantile estimation for the vector parameter. Uniform
concentration over $L=O(d)$ bins bounds the quantile errors; together with
local identifiability, this controls the vector estimation error. The
cumulative contribution of these errors is
$\widetilde O((C_{\mathrm{init},d}+C_{\alpha,d})\sqrt T)$, while the discretization, dual-update, and
stopping-time arguments proceed as in the scalar case. Details are deferred
to Appendix~\ref{app:multi-dimensional-guarantee}.

\section{Numerical Results}
\label{sec:numerical-results}
We evaluate the finite-horizon implementation in scalar and structured
two-dimensional settings against a non-contextual budgeted baseline adapted
from~\cite{wang2023learning} and a losing-only OLS baseline fitted to revealed
competing bids. The vector OLS baseline extends its scalar counterpart.

\paragraph{Experimental settings.}
Let $S\sim\mathcal U[0,1]$ and $Z\sim\mathcal U[0,1.1]$ be independent.
The scalar model uses $X=V=S$, $\alpha=0.02$, and $\Theta=[0,0.025]$.
The vector model uses $\mathbf X=(S,S^2)^\top$, $V=0.6S+0.4S^2$,
$\boldsymbol\alpha=(0.012,0.008)^\top$, and
$\Theta=[0,0.015]\times[0,0.010]$.
The competing bids are $D=\alpha X+Z$ and
$D=\langle\boldsymbol\alpha,\mathbf X\rangle+Z$, respectively.
Both use $\bar v=1$ and $B=\rho T$, with $\rho=0.05$.
It can be verified that these models satisfy the corresponding scalar and
structured-vector model assumptions.
We use $g_{\min}=10/11$, $g_{\max}=1$, and
$(L,\kappa)=(2,0.5)$ or $(3,0.03)$, respectively, sorting samples by
observed value. This budget rate constrains the mean-budget oracle.

The displayed horizons are
$T\in\{1000,\allowbreak2000,\allowbreak5000,\allowbreak10000,\allowbreak15000,\allowbreak20000\}$,
with 60 matched repetitions per configuration. All methods use the same
contexts, noise, and benchmark within each repetition. We use $p_0=0.99$,
80 bid points, and the phase schedule and budget safeguard of
Section~\ref{sec:algorithm}; the accompanying code provides the remaining
numerical details.

Figure~\ref{fig:experiment-comparison} shows that our numerical
implementation clearly outperforms both baselines in mean reward gap
across the displayed horizons in both settings.

\begin{figure}[!htbp]
\centering
\includegraphics[width=\linewidth]{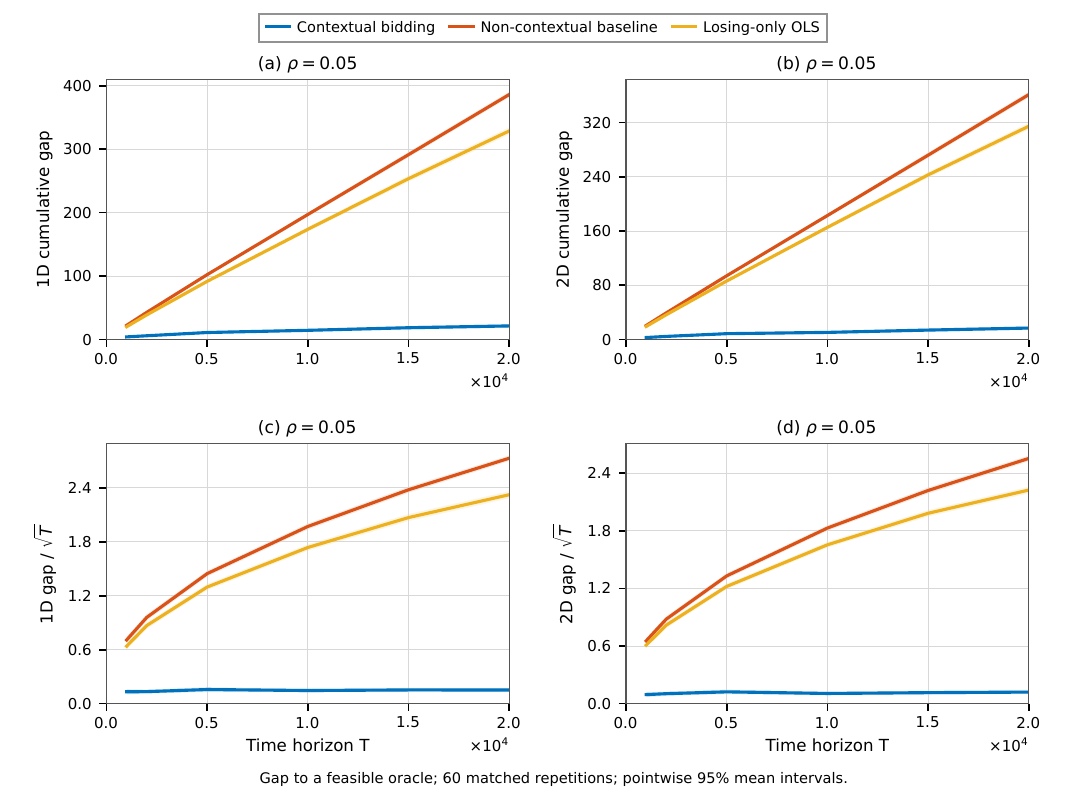}
\caption{Comparison at $\rho=0.05$: (a) scalar and (b) vector cumulative
reward gaps; (c)--(d) their $\sqrt T$-normalized counterparts.
The gap is the cumulative reward of the known-parameter mean-budget
stationary oracle, stopped when its remaining hard budget is below one,
minus the algorithm's cumulative reward. Bands are pointwise Student-$t$
95\% confidence intervals for the mean over 60 matched repetitions.
The numerical implementation uses smaller confidence intervals, recomputes
the CDF from cumulative observable data at the current parameter estimate,
and maximizes estimated reward among surviving feasible bids; the two
baselines retain their minimum-surviving-bid rules. The theoretical guarantees
apply to the algorithms specified in the analysis.}
\label{fig:experiment-comparison}
\end{figure}

\section{Conclusion}
\label{sec:conclusion}
We studied contextual first-price auctions with hard budget constraints under one-sided information feedback. Our algorithm combines quantile-based parameter estimation, bid elimination, and dual updates to learn the contextual coefficient and noise distribution while controlling expenditure. Under the stated assumptions, it achieves $\widetilde O(\sqrt T)$ regret against the optimal hard-budget policy. This result attains the same square-root dependence on the horizon, up to logarithmic factors, as established lower bounds for general first-price auction learning~\citep{han2025optimal}.

\noindent
\textbf{Limitations and Extensions.} Our theoretical analysis focuses on a scalar contextual model and a structured multi-dimensional extension where contexts lie on a monotone one-dimensional manifold. In addition, the guarantees rely on regularity, local identifiability, quantile visibility, and an effective-value margin. Relaxing these assumptions and developing algorithms for fully general high-dimensional contexts are important directions for future research, as are extensions to adversarial contexts and more limited feedback models.

\bibliographystyle{plainnat}
\bibliography{ijocreferences}

\clearpage

\begin{APPENDICES}
\renewcommand{\theHsection}{appendix.\Alph{section}}

\section{Proofs for Section \ref{sec:robust-regression}}
\label{app:robust-regression}

\subsection{Uniform concentration for uncensored residual quantiles}

\begin{lemma}[Uniform quantile concentration for uncensored residuals]
\label{lem:uniform-quantile-concentration}
Under Assumptions~\ref{ass:lipschitz} and~\ref{ass:bounded}, suppose $X$ is
atomless. Fix a complete estimation block $A_i$ of size $n_i\ge2$, and condition on
an initial history for which $\sup_{\widetilde\alpha\in\mathcal A_0}|\widetilde\alpha-\alpha|\le\beta_0$. Set $\varepsilon_i=c_{\mathrm{vc}}\sqrt{\log(2n_i/\delta)/n_i}$ and $r_{n_i}(\delta)=2\varepsilon_i/g_{\min}$. If $2\bar x\beta_0+4r_{n_i}(\delta)\le r_g$ and
$\varepsilon_i\le(1-p_0)/4$, then, with conditional probability at least $1-\delta$, $\sup_{\widetilde\alpha\in\mathcal A_0}\max_{k=1,2} |\widehat q_k^o(\widetilde\alpha)-q_k^o(\widetilde\alpha)| \le r_{n_i}(\delta).$
\end{lemma}

\begin{proof}{Proof.}
Conditional on the initial history, $\mathcal A_0$ is fixed and the latent
pairs $(X_s,Z_s)$ in $A_i$ remain i.i.d.  Write $P_{n_i}$ for their
empirical measure and $\widehat m_X$ for the empirical rank cutoff.
The sets $\{(x,z):x\le a,\ z+(\alpha-\widetilde\alpha)x\le y\}$,
indexed by $(a,y,\widetilde\alpha)\in\mathbb R^2\times\mathcal A_0$,
and the corresponding upper-group sets have finite
Vapnik--Chervonenkis (VC) dimension (polynomial growth on $N$ points).
Apply the uniform VC inequality~\citep{vapnik1971uniform} also to the context threshold sets
$\{x:x\le a\}$. Since $X$ is atomless, this controls joint probabilities and the mass
between $\widehat m_X$ and the population median. The empirical group masses are
$1/2+O(n_i^{-1})$. Dividing the joint-probability bounds by these masses
therefore gives, simultaneously for both groups,
$\sup_{\widetilde\alpha,k,y} |\widehat F_{k,\widetilde\alpha}^o(y)-F_{k,\widetilde\alpha}^o(y)| \le\varepsilon_i$
with probability at least $1-\delta$. The choice $c_{\mathrm{vc}}=1024$ covers concentration, cutoff, and rounding errors.

For inversion, let $q_0=G^{-1}(p_0)$ and $a=\bar x\beta_0$. Every residual perturbation has magnitude at most $a$, so $|q_k^o(\widetilde\alpha)-q_0|\le a$. For a residual argument within
$r_{n_i}(\delta)$ of this conditional quantile, the corresponding argument
of $g$ is within $2a+r_{n_i}(\delta)$ of $q_0$. The assumed neighborhood
condition thus gives a conditional residual density of at least
$g_{\min}$ throughout this interval. Moving a distance
$r_{n_i}(\delta)=2\varepsilon_i/g_{\min}$ to either side of the population
quantile changes the population CDF by at least $2\varepsilon_i$.
The empirical CDF is below $p_0$ at the lower endpoint and above $p_0$
at the upper endpoint, trapping its $\lceil p_0|I_k|\rceil$-th order
statistic uniformly in the candidate and group.
\hfill\Halmos
\end{proof}

\subsection{Proof of Lemma~\ref{lem:quantile-invariance}}
\label{app:quantile-invariance}

\begin{proof}{Proof.}
Work on the initial localization event and the concentration event of
Lemma~\ref{lem:uniform-quantile-concentration}. For a censored observation
$j\in I_k$, we have $d_j<b_j\le f(x_j)$. Thus, for every candidate $\widetilde\alpha\in\mathcal A_0$, $R_j^o(\widetilde\alpha) =d_j-\widetilde\alpha x_j \le f(x_j)-\alpha x_j+\bar x\beta_0$.
On the other hand, because $R^o(\widetilde\alpha)=Z+(\alpha-\widetilde\alpha)X$, $q_k^o(\widetilde\alpha) \ge G^{-1}(p_0)-\bar x\beta_0$.
Consequently, $\widehat q_k^o(\widetilde\alpha) \ge G^{-1}(p_0)-\bar x\beta_0-r_{n_i}(\delta)$. Assumption~\ref{ass:visible-quantile} and
$2\bar x\beta_0+4r_{n_i}(\delta)\le\Delta_q$ now give
$\widehat q_k^o(\widetilde\alpha) -R_j^o(\widetilde\alpha) \ge \Delta_q-2\bar x\beta_0-r_{n_i}(\delta) \ge 3r_{n_i}(\delta)>0.$
All censored residuals in $I_k$ are strictly below the
$\lceil p_0|I_k|\rceil$-th order statistic of the latent residuals. Replacing them by
$-\infty$ preserves it, giving $\widehat q_k(\widetilde\alpha)=\widehat q_k^o(\widetilde\alpha)$
for every $\widetilde\alpha\in\mathcal A_0$ and $k=1,2$.
\hfill\Halmos
\end{proof}
\subsection{Initial localization}
\label{app:initial-localization}

\begin{lemma}[Initial localization]
\label{lem:initial-localization}
The data-dependent interval in~\eqref{eq:initial-search-interval} contains
$\alpha$ with probability at least $1-\delta/4$, for every initial sample
size. In addition, if
\begin{equation}
\label{eq:scalar-design-size}
n_0\ge18\frac{\bar x^4}{\sigma_X^4}\log\frac{16}{\delta},
\end{equation}
then, with probability at least $1-\delta/2$,
$S_{xx}\ge\frac{n_0\sigma_X^2}{2}, \qquad |\widehat\alpha_0-\alpha|\le r_0, \qquad \alpha\in\mathcal A_0, \qquad \sup_{\widetilde\alpha\in\mathcal A_0}|\widetilde\alpha-\alpha| \le\beta_0\le C_{\mathrm{init}}\sqrt{\frac{\log(8/\delta)}{n_0}}.$
\end{lemma}

\begin{proof}{Proof.}
Condition on the initial contexts. If $S_{xx}>0$, the unprojected OLS error
is
$\widetilde\alpha_0-\alpha =\frac{\sum_{s\in T_0}(x_s-\bar x_{T_0})(Z_s-\mathbb EZ)}{S_{xx}}.$
The summands are conditionally independent and centered, and the sum of
their squared ranges is $\bar z^2S_{xx}$. Hoeffding's inequality~\citep{hoeffding1963probability} gives
$\mathbb P\!\left( |\widetilde\alpha_0-\alpha|> \bar z\sqrt{\frac{\log(8/\delta)}{2S_{xx}}} \,\middle|\,\{x_s\}_{s\in T_0}\right)\le\frac\delta4.$
Projection cannot increase the error and also bounds it by $\bar\alpha$.
If $S_{xx}=0$, the prescribed radius $r_0=\bar\alpha$ covers the entire
parameter interval. 

For the size bound, apply Hoeffding's inequality to the sample averages of $X$ and $X^2$. With probability at least $1-\delta/4$, $|\frac{S_{xx}}{n_0}-\sigma_X^2| \le3\bar x^2\sqrt{\frac{\log(16/\delta)}{2n_0}} \le\frac{\sigma_X^2}{2}$,
where the last inequality uses~\eqref{eq:scalar-design-size}. Intersecting with coverage gives the stated probability. On this event,
points in $\mathcal A_0$ are within $2r_0$ of $\alpha$; the lower bound on
$S_{xx}$ gives the envelope for $\beta_0$.
\hfill\Halmos
\end{proof}

\subsection{Proof of Theorem \ref{thm:main-estimation}}
\label{app:main-estimation}

\begin{proof}{Proof.}
Condition on an initial history with $\alpha\in\mathcal A_0$ and $\sup_{\widetilde\alpha\in\mathcal A_0}|\widetilde\alpha-\alpha|\le\beta_0$. Let $n_i=|A_i|$ and let $\mathcal Q_i$ be the uniform quantile-concentration
event in Lemma~\ref{lem:uniform-quantile-concentration}. Its conditional
probability is at least $1-\delta$. The neighborhood and visibility
conditions in Lemma~\ref{lem:quantile-invariance} imply, on this same event,
$\widehat q_k(\widetilde\alpha)=\widehat q_k^o(\widetilde\alpha)$
for every $\widetilde\alpha\in\mathcal A_0$ and $k=1,2$.
The quantiles are finite, so no fallback occurs.

Write $Q^o(\widetilde\alpha)=|q_1^o(\widetilde\alpha)-q_2^o(\widetilde\alpha)|$ and $\widehat Q_i(\widetilde\alpha)=|\widehat q_1(\widetilde\alpha)-\widehat q_2(\widetilde\alpha)|$.
The triangle inequality gives
$\sup_{\widetilde\alpha\in\mathcal A_0} |\widehat Q_i(\widetilde\alpha)-Q^o(\widetilde\alpha)| \le2r_{n_i}(\delta).$
At the true coefficient, the residual is $Z$, independent of the context
group, so $Q^o(\alpha)=0$.
Since $\alpha$ is feasible and $\widehat\alpha_i$ minimizes the objective,
$Q^o(\widehat\alpha_i) \le\widehat Q_i(\widehat\alpha_i)+2r_{n_i}(\delta) \le\widehat Q_i(\alpha)+2r_{n_i}(\delta) \le4r_{n_i}(\delta).$
Since $\beta_0\le r_{\mathrm{id}}$, identifiability holds on $\mathcal A_0$, giving
$\kappa|\widehat\alpha_i-\alpha| \le Q^o(\widehat\alpha_i) \le4r_{n_i}(\delta), \qquad |\widehat\alpha_i-\alpha| \le C_\alpha\sqrt{\frac{\log(2n_i/\delta)}{n_i}}.$
Also $|\widehat\alpha_i-\alpha|\le\beta_0$ because both parameters belong
to the localized search interval. This gives the minimum defining $\beta_i$.
\hfill\Halmos
\end{proof}

\section{Proofs for Section~\ref{sec:regret-analysis}}
\label{app:regret-analysis-proofs}
\subsection{Auxiliary Lemmas for Theorem~\ref{thm:main-regret}}
\label{app:main-regret-auxiliary-lemmas}

The following lemmas use the notation, radii, and constants from
Sections~\ref{sec:algorithm} and~\ref{sec:regret-analysis}, with
$\delta=T^{-2}$ in expectation bounds. Statements about updates concern
$1\le i<n$; the final, possibly truncated phase uses the preceding complete
phase's estimates. During $T_0$, set $\lambda_t=\widehat c_t=0$ to extend
the dual sequence to the full horizon. Since
$\widehat\alpha_i\in[0,\bar\alpha]$,
Assumption~\ref{ass:effective-value-margin} and
Lemma~\ref{lem:dual-bound} give $0\le q_t\le\bar v$.
Constants $C$ are universal and may vary; the proofs follow below.

\begin{lemma}[Dual variable boundedness]
\label{lem:dual-bound}
Suppose $\eta\bar v\le1$. The dual iterates generated by
Algorithm~\ref{alg:budget-bidding} satisfy, deterministically,
$0\le\lambda_t\le\bar\lambda$ for $t=1,\ldots,T+1$, where
$\bar\lambda:=\bar v/\rho$.

\end{lemma}

\begin{lemma}[Relaxed stationary benchmark]
\label{lem:relaxed-stationary-benchmark}
Let $\pi^\star$ be an optimal policy for the original hard-budget problem and
define $R_{\mathrm{rel}}^\star := T\min_{\lambda\ge0}D(\lambda)$, where $D$ is the population dual function in~\eqref{eq:population-dual}.
Since $D(\lambda)\ge\lambda\rho$ and $D(0)\le\bar v$, every minimizer belongs
to $[0,\bar v/\rho]$; hence the minimum is unchanged if it is restricted to
$[0,\bar\lambda]$. Moreover, $R(\pi^\star)\le R_{\mathrm{rel}}^\star$. Let $\lambda^\star$ be a minimizer of $D$. There exists a measurable stochastic
kernel $\mu^\star(\mathrm db\mid v,x)$ supported on $\arg\max_{0\le b\le v} (v-(1+\lambda^\star)b)G(b-\alpha x)$ such that, if $b^\star(V,X)$ is sampled from this kernel conditional on
$(V,X)$, then
$\mathbb E[c(b^\star(V,X),X)]\le\rho, \qquad \lambda^\star \bigl(\rho-\mathbb E[c(b^\star(V,X),X)]\bigr)=0.$
Consequently,
$R_{\mathrm{rel}}^\star = T\,\mathbb E[r(V,b^\star(V,X),X)].$
\end{lemma}

\begin{lemma}[Stopping-time bound]
\label{lem:stopping-time-bound}
Under the horizon conditions of Theorem~\ref{thm:main-regret}, the stopping
time $\tau$ satisfies
$\mathbb E[T-\tau] \le\frac{\bar\lambda}{\rho\eta} +C\frac{\mathcal C_{\mathrm{stat}}}{\rho}\sqrt T\log T+\frac{C}{T}.$
For $\eta=(\rho\sqrt T)^{-1}$, the right-hand side is at most $C\rho^{-1}\mathcal C_{\mathrm{stat}}\sqrt T\log T+C/T$. 
\end{lemma}

\begin{lemma}[Confidence representation]
\label{lem:confidence-representation}
There exists an event $\mathcal E_G$ with
$\mathbb P(\mathcal E_G^c)\le\delta$ such that, on $\mathcal E_G$ and the
parameter-estimation event, simultaneously for every update $1\le i<n$, every bin
$m$, and every $\widetilde b^k\in\mathcal B_{i-1}^m$, $\left| \widehat G_i(\widetilde b^k)-G(\widetilde b^k) \right| \le C\left[ \sqrt{\frac{L_T}{n_i^k\vee1}} +g_{\max}\bar x\beta_{i-1} \right],$
and $| \widehat r_i(q^m,\widetilde b^k) -r_q(q^m,\widetilde b^k) | \le\omega_i^m$. Moreover, uniformly over all $x\in\operatorname{supp}(X)$ satisfying $0\le\widetilde b^k+\widehat\alpha_i x\le f(x)$, $| \widehat c_i(\widetilde b^k;x) -c(\widetilde b^k+\widehat\alpha_i x,x) | \le\omega_i^m$.
Accordingly, confidence elimination uses tolerance $2\omega_i^m$. The oracle
candidates are preserved by confidence elimination and monotone closure:
$\widetilde b_m^\star\in\mathcal B_i^m$ for $m=0,\ldots,M$.
For every $i\ge2$, every $t\in A_i\cup U_i$, and every comparator
$b\in[0,v_t]$, the selected bid satisfies
\begin{align}
r(v_t,b,x_t)-r(v_t,b_t,x_t)
&\le
C(1+\bar\lambda)
\left[
\omega_{i-1}^{m(t)}+\varepsilon_{\mathrm{grid}}
\right]
\notag\\
&\quad+
\lambda_t
\left[
c(b,x_t)-\widehat c_t
\right].
\label{eq:one-step-lagrangian-comparison}
\end{align}
For a comparator with positive Lagrangian reward, $z_t\ge0$ implies
$b\ge\alpha x_t$. Its true shifted bid $b-\alpha x_t$ is therefore
nonnegative. Comparing this bid with the shifted grid and accounting for
the estimation error requires no negative grid point.

\end{lemma}

\begin{lemma}[Confidence-sum bound]
\label{lem:confidence-sum-bound}
Under the assumptions and horizon conditions of
Theorem~\ref{thm:main-regret},
$\mathbb E\!\left[ \sum_{i=2}^n\sum_{t\in A_i\cup U_i}\omega_{i-1}^{m(t)} \right] \le C\mathcal C_{\mathrm{stat}}\sqrt T\log T.$
\end{lemma}

\begin{lemma}[Dual-term bound]
\label{lem:c-bound}
Let $b_t^\star=b^\star(v_t,x_t)$, including the benchmark randomization
specified in Lemma~\ref{lem:relaxed-stationary-benchmark}. Under the dual update $\lambda_{t+1} = \max\{0,\lambda_t+\eta(\widehat c_t-\rho)\}$, we have
$\mathbb E\!\left[ \sum_{t=1}^T \lambda_t \bigl(c(b_t^\star,x_t)-\widehat c_t\bigr) \right] \le C\eta T\bar v^2.$
In particular, for $T\ge(\bar v/\rho)^2$ and
$\eta=(\rho\sqrt T)^{-1}$, the right-hand side is at most
$C\bar v^2\sqrt T/\rho$. 
\end{lemma}
\subsection{Sufficient horizon conditions for Theorem~\ref{thm:main-regret}}
\label{app:main-regret-horizon}

For the horizon conditions, recall $H_0=\lceil\sqrt T\rceil$,
$n_0=2H_0$ for the horizons considered here, and $\delta=T^{-2}$. Set
$\ell_T:=\log(8T^4), \qquad \overline\beta_0(T):=C_{\mathrm{init}}\sqrt{\frac{\log(8/\delta)}{n_0}}, \qquad \overline r(T):=C_q\sqrt{\frac{\ell_T}{H_0}}.$
Require $T$ to satisfy
\begin{equation}
\label{eq:explicit-large-horizon}
\begin{gathered}
T\ge\max\{16,(\bar v/\rho)^2\},
\qquad
n_0\ge18\frac{\bar x^4}{\sigma_X^4}\log\frac{16}{\delta},\\
\overline\beta_0(T)\le r_{\mathrm{id}},
\qquad
2\bar x\overline\beta_0(T)+4\overline r(T)
\le\min\{\Delta_q,r_g\},\\
c_{\mathrm{vc}}\sqrt{\frac{\ell_T}{H_0}}\le\frac{1-p_0}{4}.
\end{gathered}
\end{equation}
These conditions ensure budget feasibility, the uncapped step size,
variance concentration, and valid localization, density, and visibility
neighborhoods, with empirical error controlled at $p_0$.
For fixed parameters, choose $T_\star$ so they hold for every $T\ge T_\star$.

\subsection{Proof of Lemma~\ref{lem:dual-bound}}
\label{app:dual-bound}

\begin{proof}{Proof.}
Outside the initial exploration block, the dual update is
$\lambda_{t+1} = \max\{0,\lambda_t+\eta(\widehat c_t-\rho)\}, \qquad \widehat c_t = \widehat c_{i-1}(\widetilde b_t;x_t), \quad t\in A_i\cup U_i.$
During $T_0$, we use the convention $\lambda_t=\widehat c_t=0$.

If the empty-set fallback is triggered, then $\widehat c_t=0$. Otherwise, by construction,
$\widehat c_t = (\widetilde b_t+\widehat\alpha_{i-1}x_t) \widehat G_{i-1}(\widetilde b_t) = b_t\widehat G_{i-1}(\widetilde b_t).$
Since $0\le\widehat G_{i-1}\le1$ and the selected bid is feasible,
\begin{equation}
\label{eq:dual-bound-cost-control}
0\le\widehat c_t\le b_t
\le\frac{v_t}{1+\lambda_t}
\le\frac{\bar v}{1+\lambda_t}.
\end{equation}
The bound is deterministic.

Let $\lambda_{\mathrm{th}}:=\bar v/\rho-1$, which is nonnegative because
$\rho\le\bar v$. If
$\lambda_t\ge\lambda_{\mathrm{th}}$, then $\widehat c_t \le \frac{\bar v}{1+(\bar v/\rho-1)} =\rho$, and hence $\lambda_{t+1}\le\lambda_t$. If
$\lambda_t<\lambda_{\mathrm{th}}$, then $\lambda_{t+1} \le \lambda_t+\eta\widehat c_t \le \lambda_{\mathrm{th}}+\eta\bar v \le \frac{\bar v}{\rho} = \bar\lambda$, where the last inequality uses $\eta\bar v\le1$. Since the dual variable is initialized at
$\lambda_{|T_0|+1}=0$, induction gives $0\le\lambda_t\le\bar\lambda$
for $t=1,\ldots,T$.
 
\hfill\Halmos
\end{proof}

\subsection{Proof of Lemma~\ref{lem:relaxed-stationary-benchmark}}
\label{app:relaxed-stationary-benchmark}

\begin{proof}{Proof.}
Let $\Pi_{\mathrm{all}}$ be the class of nonanticipating policies subject
only to $0\le b_t\le v_t$. For $\pi\in\Pi_{\mathrm{all}}$, define
$R(\pi):=\mathbb E^\pi[\sum_{t=1}^T r(v_t,b_t,x_t)]$ and
$C(\pi):=\mathbb E^\pi[\sum_{t=1}^T c(b_t,x_t)]$.
The hard budget gives $C(\pi^\star)\le\rho T$, so
$R(\pi^\star)\le R(\pi^\star)-\lambda(C(\pi^\star)-\rho T)$ for
every $\lambda\ge0$. The supremum of the latter Lagrangian over
$\Pi_{\mathrm{all}}$ equals $TD(\lambda)$: the i.i.d. pairs $(V_t,X_t)$
and absence of intertemporal constraints make it separable, and the
measurable maximum theorem~\citep[Chapter~18]{aliprantis2006infinite}
provides pointwise maximizers. Thus,
\begin{equation}
\label{eq:relaxed-benchmark-weak-duality}
R(\pi^\star)\le T D(\lambda),
\qquad \lambda\ge0.
\end{equation}
The continuous convex function $D$ satisfies $D(\lambda)\ge\lambda\rho$
by choosing $b=0$, whereas $D(0)\le\bar v$. Hence its minimum is attained
in $[0,\bar v/\rho]=[0,\bar\lambda]$, and
\eqref{eq:relaxed-benchmark-weak-duality} gives
$R(\pi^\star)\le T\min_{\lambda\ge0}D(\lambda)=R_{\mathrm{rel}}^\star$.

Fix a minimizer $\lambda^\star$ and write
$\mathcal M_\lambda(v,x):=\arg\max_{0\le b\le v}
(v-(1+\lambda)b)G(b-\alpha x)$.
The measurable maximum theorem gives measurable selections from
$\mathcal M_\lambda$, and Danskin's theorem~\citep{danskin1967theory} gives
\begin{equation}
\label{eq:dual-subdifferential-selections}
\partial D(\lambda)
=
\operatorname{co}
\left\{
\rho-\mathbb E[c(b(V,X),X)]:
b\text{ is a measurable selection of }\mathcal M_\lambda
\right\}.
\end{equation}
Convex optimality gives $0\in\partial D(\lambda^\star)$ if
$\lambda^\star>0$, and $\partial D(0)\cap[0,\infty)\neq\varnothing$
if $\lambda^\star=0$. Accordingly, choose
$g^\star\in\partial D(\lambda^\star)$ with $g^\star=0$ in the first
case and $g^\star\ge0$ in the second. Since these subgradients are scalar,
\eqref{eq:dual-subdifferential-selections} represents $g^\star$ as a
convex combination of at most two measurable maximizer selections.
Randomizing between them defines a measurable kernel
$\mu^\star(\mathrm db\mid v,x)$ supported on
$\mathcal M_{\lambda^\star}(v,x)$. A draw
$b^\star(V,X)\sim\mu^\star(\cdot\mid V,X)$ therefore satisfies
$\mathbb E[c(b^\star(V,X),X)]=\rho-g^\star\le\rho$ and
$\lambda^\star(\rho-\mathbb E[c(b^\star(V,X),X)])=0$.
For the subsequent analysis, draw $b_t^\star$ from this same kernel
independently of the algorithmic history and of the other benchmark draws,
conditional on $(V_t,X_t)$.

The kernel is supported on Lagrangian maximizers, so
$D(\lambda^\star)=\mathbb E[r(V,b^\star(V,X),X)]
+\lambda^\star(\rho-\mathbb E[c(b^\star(V,X),X)])$.
Complementary slackness cancels the second term, yielding
$R_{\mathrm{rel}}^\star=TD(\lambda^\star)
=T\mathbb E[r(V,b^\star(V,X),X)]$.
\hfill\Halmos
\end{proof}

\subsection{Proof of Lemma~\ref{lem:stopping-time-bound}}
\label{app:stopping-time-bound}
\begin{proof}{Proof.}
Let
$\tau := \inf\{t\ge1:B_{t+1}<\bar v\}\wedge T$
be the last round played before the budget safeguard is triggered, and let
\(I_0:=T_0\cup A_1\cup U_1\). We consider the formal continuation of the
algorithm after \(\tau\), as in the regret analysis.
Define $C_t:=b_t\mathbf 1\{b_t>d_t\}$ and
$c_t:=c(b_t,x_t)=b_tG(b_t-\alpha x_t)$.
Let \(\mathcal E\) be the common confidence event.

For every post-exploration round, the dual update gives $\lambda_{t+1} = \max\{0,\lambda_t+\eta(\widehat c_t-\rho)\} \ge \lambda_t+\eta(\widehat c_t-\rho)$. Since \(\lambda_{|T_0|+1}=0\) and
\(0\le\lambda_t\le\bar\lambda\) by
Lemma~\ref{lem:dual-bound}, summing up to \(\tau\) yields
$\sum_{\substack{t\le\tau\\t\notin T_0}}\widehat c_t \le \rho\tau+\frac{\bar\lambda}{\eta}.$
On \(\mathcal E\), Lemma~\ref{lem:confidence-representation} gives, for
\(i\ge2\) and \(t\in A_i\cup U_i\), $|c_t-\widehat c_t| \le \omega_{i-1}^{m(t)}$. The rounds in \(I_0\) contribute at most \(C\bar v\sqrt T\). Therefore,
\begin{align*}
\sum_{t\le\tau}c_t
\le
\rho\tau+\frac{\bar\lambda}{\eta}
+C\bar v\sqrt T
+
\sum_{i=2}^n\sum_{t\in A_i\cup U_i}
\omega_{i-1}^{m(t)}
\qquad\text{on }\mathcal E.
\tag{A}
\end{align*}
Define
$M_s := \sum_{t=1}^s(C_t-c_t).$
Conditioned on the history immediately before the auction outcome,
\(C_t-c_t\) has mean zero and absolute value at most \(\bar v\). Hence
\(\{M_s\}_{s=1}^T\) is a martingale and
$\mathbb E[M_T^2]\le T\bar v^2.$
Doob's \(L^2\) inequality~\citep{doob1953stochastic} gives
\begin{align*}
\mathbb E\!\left[\sup_{s\le T}|M_s|\right]
\le
2\sqrt{\mathbb E[M_T^2]}
\le
2\bar v\sqrt T.
\tag{B}
\end{align*}
If \(\tau<T\), then \(B_{\tau+1}<\bar v\), and hence
$\sum_{t\le\tau}C_t>\rho T-\bar v.$
Together with \(M_\tau=\sum_{t\le\tau}(C_t-c_t)\) and \((A)\), this
gives on \(\mathcal E\), including the trivial case \(\tau=T\),
$\rho(T-\tau) \le \bar v+\frac{\bar\lambda}{\eta} +C\bar v\sqrt T + \sum_{i=2}^n\sum_{t\in A_i\cup U_i} \omega_{i-1}^{m(t)} + \sup_{s\le T}|M_s|.$
Taking expectations and applying \((B)\) and
Lemma~\ref{lem:confidence-sum-bound} gives
$\mathbb E[(T-\tau)\mathbf 1_{\mathcal E}] \le\frac{\bar\lambda}{\rho\eta} +C\frac{\mathcal C_{\mathrm{stat}}}{\rho}\sqrt T\log T.$
Use $\mathcal C_{\mathrm{stat}}\ge\bar v$ to absorb the initial and martingale terms.
Finally,
$\mathbb E[(T-\tau)\mathbf 1_{\mathcal E^c}] \le T\mathbb P(\mathcal E^c) \le \frac{C}{T},$
since \(\delta=T^{-2}\). Summing the two event bounds gives the stated inequality.
For $T\ge(\bar v/\rho)^2$, the first term equals
$\bar\lambda\sqrt T=(\bar v/\rho)\sqrt T$, which is absorbed by the
second. 
\hfill\Halmos
\end{proof}

\subsection{Proof of Lemma~\ref{lem:confidence-representation}}
\label{app:confidence-representation}
\begin{proof}{Proof.}
Fix an update $i$, a bin $m$, and a candidate
$\widetilde b^k\in\mathcal B_{i-1}^m$. For $s\in U_i$, write
$I_s^k=\mathbf 1\{\widetilde b_s\le\widetilde b^k\}$ and
$n_i^k=\sum_{s\in U_i}I_s^k$. Recall the observable
$Y_s^k=\mathbf 1\{b_s>d_s\}+\mathbf 1\{b_s\le d_s\}
\mathbf 1\{\widetilde d_s\le\widetilde b^k\}$.
When $I_s^k=1$, winning implies
$d_s<b_s\le\widetilde b^k+\widehat\alpha_{i-1}x_s$; after a loss,
$\widetilde d_s$ is observed. Thus
$I_s^kY_s^k=I_s^k\mathbf 1\{\widetilde d_s\le\widetilde b^k\}$,
the latent representation of~\eqref{eq:G-estimator}.

Let $\mathcal G_s$ contain the past history, current context and value,
and selected bid, immediately before the auction outcome. Then $I_s^k$
is $\mathcal G_s$-measurable and
$\mathbb E[\mathbf 1\{\widetilde d_s\le\widetilde b^k\}\mid\mathcal G_s]
=G(\widetilde b^k+(\widehat\alpha_{i-1}-\alpha)x_s)$.
Consequently,
$I_s^k[Y_s^k-G(\widetilde b^k+(\widehat\alpha_{i-1}-\alpha)x_s)]$
is a bounded martingale difference with predictable quadratic variation
at most $n_i^k$. Freedman's inequality~\citep{freedman1975tail}, unioned
over all phases, grid points, and possible counts, defines an event
$\mathcal E_G$ with $\mathbb P(\mathcal E_G^c)\le\delta$.
On its intersection with the parameter-estimation event, Lipschitz
continuity of $G$ gives
\begin{equation}
\left|
\widehat G_i(\widetilde b^k)-G(\widetilde b^k)
\right|
\le
C
\left[
\sqrt{\frac{L_T}{n_i^k\vee1}}
+
g_{\max}\bar x\beta_{i-1}
\right].
\label{eq:G-confidence-bound}
\end{equation}
The sampling coefficient is at most $8$ and the bias coefficient $1$:
for $n_i^k\ge L_T$, absorb Freedman's linear term into
$\sqrt{L_T/n_i^k}$; for smaller counts, including zero, use the
probability-difference bound $1$. Since $n_i^k\ge N_i^m$ and
$|q^m-\widetilde b^k|\le\bar v$, the definition
$r_q(q,\widetilde b)=(q-\widetilde b)G(\widetilde b)$ yields
\begin{equation}
\left|
\widehat r_i(q^m,\widetilde b^k)
-
r_q(q^m,\widetilde b^k)
\right|
\le
\omega_i^m.
\label{eq:proxy-reward-confidence}
\end{equation}
For a feasible $b_i^k(x)=\widetilde b^k+\widehat\alpha_i x\in[0,f(x)]$,
$c(b_i^k(x),x)=b_i^k(x)G(\widetilde b^k+(\widehat\alpha_i-\alpha)x)$.
Thus the cost error is at most
$\bar v[|\widehat G_i(\widetilde b^k)-G(\widetilde b^k)|
+g_{\max}\bar x\beta_i]$, and~\eqref{eq:G-confidence-bound} gives
\begin{align*}
\left|
\widehat c_i(\widetilde b^k;x)
-
c(b_i^k(x),x)
\right|
\le
\omega_i^m.
\tag{C}
\end{align*}
Both parameter shifts $\beta_{i-1}+\beta_i$ are included in
$\omega_i^m$; the stated coefficients are covered by $C_\omega=64$.
The cap in~\eqref{eq:confidence-radius} is valid: proxy rewards lie in
$[-\bar v,\bar v]$, and both feasible costs lie in $[0,\bar v]$.

For oracle retention, let
$\widetilde b_m^\star=\min\arg\max_{\widetilde b^k\in\widetilde{\mathcal B}}
r_q(q^m,\widetilde b^k)$. Since $G(0)=0$,
$0\le\widetilde b_m^\star\le q^m$.
Starting from $\mathcal B_0^m=\widetilde{\mathcal B}$, assume retention
at the preceding update. Equation~\eqref{eq:proxy-reward-confidence}
shows that confidence elimination retains the oracle and bounds every
survivor's true proxy gap by $4\omega_i^m$. Induct also over the filtering
bins: retention in lower bins and
Lemma~\ref{lem:residual-monotonicity} give
$\max_{\ell<m}\inf\mathcal B_i^\ell
\le\max_{\ell<m}\widetilde b_\ell^\star\le\widetilde b_m^\star$.
The monotone closure therefore retains the oracle as well, so neither
the empty-action fallback nor the empty-closure reset is used on the
common good event. For every $\widetilde b\in\mathcal B_i^m$,
\begin{equation}
r_q(q^m,\widetilde b_m^\star)
-
r_q(q^m,\widetilde b)
\le
4\omega_i^m.
\label{eq:active-proxy-gap}
\end{equation}

For the one-step comparison, fix $i\ge2$, $t\in A_i\cup U_i$, and set
$q_t^\circ=v_t/(1+\lambda_t)-\alpha x_t$ and
$q_t=v_t/(1+\lambda_t)-\widehat\alpha_{i-1}x_t$.
On the common event, $|q_t-q_t^\circ|\le\bar x\beta_{i-1}$.
Clipping, Lemma~\ref{lem:dual-bound}, and
Assumption~\ref{ass:effective-value-margin} give
$0\le\gamma_qv_t\le q_t\le\bar v$, hence
$0\le q_t-q^{m(t)}\le h_q$.
The retained oracle satisfies
$\widetilde b_{m(t)}^\star\le q^{m(t)}\le q_t$.
The smallest active bid is therefore feasible and selected:
$\widetilde b_t=\inf\mathcal B_{i-1}^{m(t)}\le q_t$.

For a comparator $b\in[0,v_t]$, write
$L_t(b)=r(v_t,b,x_t)-\lambda_t c(b,x_t)$.
If $L_t(b)\le0$, comparison follows from $L_t(b_t)\ge0$.
Otherwise, $z_t\ge0$ implies
$\alpha x_t\le b\le v_t/(1+\lambda_t)$, so its true shifted bid
$\widetilde b^\alpha(b)=b-\alpha x_t$ belongs to $[0,\bar v]$.
Grid approximation, the two bounds on $q_t$, and oracle optimality give
$r_q(q^{m(t)},\widetilde b_{m(t)}^\star)
\ge r_q(q_t^\circ,\widetilde b^\alpha(b))
-C(\bar x\beta_{i-1}+\varepsilon_{\mathrm{grid}})$.
Since $\widetilde b_t\in\mathcal B_{i-1}^{m(t)}$,
\eqref{eq:active-proxy-gap} at index $i-1$ then gives
$r_q(q_t,\widetilde b_t)
\ge r_q(q_t^\circ,\widetilde b^\alpha(b))
-C[\omega_{i-1}^{m(t)}+\bar x\beta_{i-1}
+\varepsilon_{\mathrm{grid}}]$.

When the cap is inactive, $\omega_{i-1}^{m(t)}$ contains both
$\beta_{i-2}$ and $\beta_{i-1}$ and covers the current plug-in error.
Using
$L_t(b)=(1+\lambda_t)r_q(q_t^\circ,\widetilde b^\alpha(b))$,
$L_t(b_t)=(1+\lambda_t)(q_t-\widetilde b_t)
G(\widetilde b_t+(\widehat\alpha_{i-1}-\alpha)x_t)$,
and Lipschitz continuity of $G$ yields
$L_t(b_t)\ge L_t(b)-C(1+\bar\lambda)
[\omega_{i-1}^{m(t)}+\varepsilon_{\mathrm{grid}}]$.
If the cap is active, the same inequality follows from
$L_t(b)\le\bar v$ and $L_t(b_t)\ge0$.
Finally, $(C)$ at index $i-1$ gives
$|c(b_t,x_t)-\widehat c_t|\le\omega_{i-1}^{m(t)}$.
Rearranging gives~\eqref{eq:one-step-lagrangian-comparison} for every
comparator $b\in[0,v_t]$.
\hfill\Halmos
\end{proof}

\subsection{Proof of Lemma~\ref{lem:confidence-sum-bound}}
\label{app:confidence-sum-bound}
For $i\ge2$, let $\mathcal H_i:=U_{i-1}$, $H_i:=|\mathcal H_i|$, and $N_{i-1}(t):=N_{i-1}^{m(t)}$.
\begin{lemma}[Scaled-rank lower bound]
\label{lem:rank-lower-bound}
On the common confidence event, for every \(i\ge2\) and
\(t\in A_i\cup U_i\),
$N_{i-1}(t) \ge R_{i,t}-1, \qquad R_{i,t} := 1+ \sum_{s\in\mathcal H_i} \mathbf 1\{v_s\le\gamma_qv_t\}.$
\end{lemma}
\begin{proof}{Proof.}
For $s\in\mathcal H_i$, nonnegativity of the clipped estimate and dual variable gives $q_s=v_s/(1+\lambda_s)-\widehat\alpha_{i-2}x_s\le v_s$. Lemma~\ref{lem:dual-bound}, clipping $\widehat\alpha_{i-1}\le\bar\alpha$, and Assumption~\ref{ass:effective-value-margin} give $q_t=v_t/(1+\lambda_t)-\widehat\alpha_{i-1}x_t\ge\gamma_qv_t$. Thus $v_s\le\gamma_qv_t$ implies $q_s\le q_t$. By Lemma~\ref{lem:selected-shifted-bid-monotonicity}, $\widetilde b_s\le\widetilde b$ for every $\widetilde b\in\mathcal B_{i-2}^{m(t)}$. Each such $s$ therefore contributes to every count $n_{i-1}^k$ entering $N_{i-1}(t)$.
\hfill\Halmos
\end{proof}
\begin{lemma}[Scaled-rank counting bound]
\label{lem:rank-counting-bound}
Under Assumption~\ref{ass:value-density},
$\mathbb E \left[ \sum_{i=2}^n \sum_{t\in A_i\cup U_i} \frac1{\sqrt{R_{i,t}}} \right] \le\frac{C}{\sqrt{\kappa_V}}\sqrt T.$
\end{lemma}
\begin{proof}{Proof.}
Set $a:=\gamma_q\in(0,1]$. The value-density bounds give $F_V(av)\ge\underline h_Vav\ge\kappa_VF_V(v)$, with $\kappa_V:=\underline h_Va/\overline h_V$. The blocks are deterministic and the values i.i.d., so, conditional only on $V_t=v$, $R_{i,t}-1\sim\operatorname{Binomial}(H_i,F_V(av))$. We do not condition their law on the confidence event. For $S\sim\operatorname{Binomial}(H,p)$ with $p>0$, $\mathbb E[1/(S+1)]=[1-(1-p)^{H+1}]/[(H+1)p]\le1/[(H+1)p]$. Cauchy--Schwarz therefore gives
$\mathbb E\!\left[\frac1{\sqrt{R_{i,t}}}\,\middle|\,V_t=v\right] \le\frac1{\sqrt{(H_i+1)F_V(av)}} \le\frac1{\sqrt{\kappa_V(H_i+1)F_V(v)}}.$
Since $F_V(V_t)\sim\operatorname{Unif}(0,1)$, integrating yields $\mathbb E[R_{i,t}^{-1/2}]\le2/\sqrt{\kappa_V(H_i+1)}$. Hence
$\mathbb E\!\left[\sum_{t\in A_i\cup U_i}\frac1{\sqrt{R_{i,t}}}\right] \le\frac{2|A_i\cup U_i|}{\sqrt{\kappa_V(H_i+1)}} \le\frac{C}{\sqrt{\kappa_V}}\sqrt{|A_i\cup U_i|},$
where the doubling schedule gives the last inequality. Summing geometrically increasing phases, including the truncated final phase, gives $\sum_{i=2}^n\sqrt{|A_i\cup U_i|}\le C\sqrt T$.
\hfill\Halmos
\end{proof}
\begin{proof}{Proof of Lemma~\ref{lem:confidence-sum-bound}}
On $\mathcal E$, Lemma~\ref{lem:rank-lower-bound} gives $N_{i-1}^{m(t)}\ge R_{i,t}-1$, so $(N_{i-1}^{m(t)}\vee1)^{-1/2}\le\sqrt2R_{i,t}^{-1/2}$. Keeping both parameter radii in~\eqref{eq:confidence-radius},
$\omega_{i-1}^{m(t)} \le C\left[ \bar v\sqrt{L_T}\,R_{i,t}^{-1/2} +(1+\bar v g_{\max})\bar x(\beta_{i-2}+\beta_{i-1}) \right].$
The comparison is pathwise on $\mathcal E$. Dropping its indicator for the nonnegative rank terms permits Lemma~\ref{lem:rank-counting-bound} under the original i.i.d. law, giving a sampling contribution at most $C\bar v\sqrt{TL_T/\kappa_V}$. On $\mathcal E_0$, initial localization gives $\beta_0\le C_{\mathrm{init}}\sqrt{\log(8/\delta)/n_0}$; subsequent radii satisfy $\beta_j\le C_\alpha\sqrt{\ell_T/|A_j|}$. Each phase, including the truncated final phase, has at most four times the preceding update-block size. Geometric summation therefore gives, on $\mathcal E_0$,
$\sum_{i=2}^n|A_i\cup U_i|(\beta_{i-2}+\beta_{i-1}) \le C(C_{\mathrm{init}}+C_\alpha)\sqrt{T\ell_T}.$
Since $L_T,\ell_T=O(\log T)$, these contributions restricted to $\mathcal E$ total at most $C\mathcal C_{\mathrm{stat}}\sqrt T\log T$. For the complement, the deterministic cap $\omega_i^m\le2\bar v$ gives $\mathbb E[\I_{\mathcal E^c}\sum_{i=2}^n\sum_{t\in A_i\cup U_i}\omega_{i-1}^{m(t)}]\le2\bar vT\mathbb P(\mathcal E^c)\le4\bar v/T$, completing the unconditional bound.
\hfill\Halmos
\end{proof}

\subsection{Proof of Lemma~\ref{lem:c-bound}}
\label{app:c-bound}
\begin{proof}{Proof.}
With $\lambda_t=\widehat c_t=0$ during $T_0$, the update is $\lambda_{t+1}=\Pi_{\mathbb R_+}[\lambda_t-\eta(\rho-\widehat c_t)]$ for every $t$. Non-expansiveness of projection~\citep{zinkevich2003online} gives $\lambda_{t+1}^2\le(\lambda_t-\eta(\rho-\widehat c_t))^2$. Expanding and summing,
$\sum_{t=1}^T\lambda_t(\rho-\widehat c_t) \le\frac{\lambda_1^2-\lambda_{T+1}^2}{2\eta} +\frac{\eta}{2}\sum_{t=1}^T(\rho-\widehat c_t)^2.$
Since $\lambda_1=0$, $0\le\rho\le\bar v$, and $0\le\widehat c_t\le\bar v$,
\begin{align*}
\sum_{t=1}^T
\lambda_t(\widehat c_t-\rho)
\ge
-C\eta T\bar v^2.
\tag{E}
\end{align*}
Let $\mathcal F_t^{\mathrm{hist}}$ denote the history $\sigma$-algebra
after round $t$ of the formal continuation, generated by all observations
and algorithm randomization through that round.
Let $b_t^\star\sim\mu^\star(\cdot\mid v_t,x_t)$ be the fresh benchmark draw from Lemma~\ref{lem:relaxed-stationary-benchmark}. The current pair $(v_t,x_t)$ and fresh randomization are independent of $\mathcal F_{t-1}^{\mathrm{hist}}$, so benchmark feasibility gives $\mathbb E[c(b_t^\star,x_t)\mid\mathcal F_{t-1}^{\mathrm{hist}}]=\mathbb E[c(b^\star(V,X),X)]\le\rho$. Since $\lambda_t$ is nonnegative and $\mathcal F_{t-1}^{\mathrm{hist}}$-measurable,
\begin{align*}
\mathbb E
\left[
\lambda_t
\bigl(\rho-c(b_t^\star,x_t)\bigr)
\right]
\ge
0.
\tag{F}
\end{align*}
Sum the identity $\lambda_t(\widehat c_t-c(b_t^\star,x_t))=\lambda_t(\widehat c_t-\rho)+\lambda_t(\rho-c(b_t^\star,x_t))$, take expectations under the original law, and apply $(E)$ and $(F)$ to obtain the bound. For $T\ge(\bar v/\rho)^2$, substituting $\eta=(\rho\sqrt T)^{-1}$ gives the stated specialization.
\hfill\Halmos
\end{proof}

\section{Structured Multi-dimensional Assumptions and Proof}
\label{app:multi-dimensional-guarantee}

For Section~\ref{sec:multi-extension}, quantile estimation sorts observed
$v_t$, agreeing with the latent order under Assumption~\ref{ass:multi-context};
bidding and sample counting use $q$.

In this appendix, use the reward and cost functions
$r(v,b,\mathbf x)$ and $c(b,\mathbf x)$ from
\eqref{eq:reward-cost-functions}, replacing $\alpha x$ by
$\langle\boldsymbol\alpha,\mathbf x\rangle$.
\begin{assumption}[Structured vector contexts]
\label{ass:multi-context}
There is a latent variable $S\in[0,1]$ and a coordinatewise strictly
increasing map $h:[0,1]\to\mathbb R_+^d$ such that
$\mathbf X=h(S)$. The function $\phi(s):=f(h(s))$ is strictly increasing,
$\|\mathbf X\|_2\le\bar x$, and $V=f(\mathbf X)\in[0,\bar v]$. The unknown
parameter belongs to a known compact convex set
$\Theta\subset\mathbb R_+^d$, and $Z\in[0,\bar z]$ is independent of $\mathbf X$. Moreover, for some $\kappa_h>0$, $\Sigma_h:=\operatorname{Cov}(h(S))\succeq\kappa_h I_d.$
The pairs $(S_t,Z_t)$ are i.i.d. across rounds.
\end{assumption}

Value sorting needs no manifold, but the manifold aligns it with the
latent order. Covariance nondegeneracy, possible for $h(s)=(s,s^2)$,
supports vector OLS; quantile identifiability is separately required in
Assumption~\ref{ass:multi-identifiability}.

\begin{assumption}[Multi-dimensional effective-value margin]
\label{ass:multi-effective-value-margin}
There is a constant $\gamma_q\in(0,1]$ such that
$\frac{f(\mathbf x)}{1+\bar\lambda} - \sup_{\boldsymbol\theta\in\Theta} \langle\boldsymbol\theta,\mathbf x\rangle \ge \gamma_q f(\mathbf x), \qquad \mathbf x\in\operatorname{supp}(\mathbf X),$
where $\bar\lambda=\bar v/\rho$.
\end{assumption}

\begin{assumption}[Multi-dimensional visibility]
\label{ass:multi-visible-quantile}
There is a constant $\Delta_q>0$ such that, for every $\mathbf x\in\operatorname{supp}(\mathbf X)$, $G^{-1}(p_0) \ge f(\mathbf x)-\langle\boldsymbol\alpha,\mathbf x\rangle+\Delta_q$.
\end{assumption}

Let $F_V$ denote the distribution function of $V$ and set
$a_\ell=F_V^{-1}(\ell/L)$, $\ell=0,\ldots,L$. Write
$\mathcal J_\ell=(a_{\ell-1},a_\ell]$ for the population value-quantile bins,
with the evident convention at the left endpoint.

\begin{assumption}[Multi-dimensional identifiability]
\label{ass:multi-identifiability}
Let $L\ge d+1$. For $\boldsymbol\theta\in\Theta$, let
$q_\ell^o(\boldsymbol\theta)$ be the population $p_0$-quantile of
$D-\langle\boldsymbol\theta,\mathbf X\rangle$ conditional on
$V\in\mathcal J_\ell$, and define
$\bar q^o(\boldsymbol\theta) =L^{-1}\sum_{\ell=1}^Lq_\ell^o(\boldsymbol\theta), \qquad Q^o(\boldsymbol\theta) = \left\{ L^{-1}\sum_{\ell=1}^L \bigl(q_\ell^o(\boldsymbol\theta)-\bar q^o(\boldsymbol\theta)\bigr)^2 \right\}^{1/2}.$
There are fixed constants $\kappa,r_{\mathrm{id}}>0$ such that $Q^o(\boldsymbol\theta)-Q^o(\boldsymbol\alpha) \ge \kappa\|\boldsymbol\theta-\boldsymbol\alpha\|_2$ whenever $\|\boldsymbol\theta-\boldsymbol\alpha\|_2\le r_{\mathrm{id}}$. \end{assumption}

\paragraph{Initial localization and quantile estimator.}
Let $n_0=|T_0|$, and let $\bar{\mathbf X}_{T_0}$ and $\bar D_{T_0}$ denote
the initial sample means of the contexts and competing bids. Whenever the
centered empirical Gram matrix is invertible, define
\begin{equation}
\label{eq:multi-initial-estimate}
\begin{aligned}
\widetilde{\boldsymbol\alpha}_0
&=
\left[
\sum_{s\in T_0}
(\mathbf X_s-\bar{\mathbf X}_{T_0})
(\mathbf X_s-\bar{\mathbf X}_{T_0})^\top
\right]^{-1}
\sum_{s\in T_0}
(\mathbf X_s-\bar{\mathbf X}_{T_0})(D_s-\bar D_{T_0}),\\
\widehat{\boldsymbol\alpha}_0
&=\Pi_\Theta(\widetilde{\boldsymbol\alpha}_0).
\end{aligned}
\end{equation}
If this centered Gram matrix, denoted by $S_0$, is singular, take a fixed
point of $\Theta$ instead. Let $D_\Theta=\operatorname{diam}(\Theta)$ and
define
\begin{equation}
\label{eq:multi-initial-radius}
\begin{aligned}
r_{0,d}&=
\begin{cases}
\displaystyle\min\left\{D_\Theta,
\bar z\sqrt{\frac{2[d\log5+\log(8/\delta)]}{\lambda_{\min}(S_0)}}\right\},
&\lambda_{\min}(S_0)>0,\\[2pt]
D_\Theta,&\lambda_{\min}(S_0)=0,
\end{cases}\\
\beta_{0,d}&=2r_{0,d},\qquad
\mathcal A_0=\{\boldsymbol\theta\in\Theta:
\|\boldsymbol\theta-\widehat{\boldsymbol\alpha}_0\|_2\le r_{0,d}\}.
\end{aligned}
\end{equation}
This observed-design radius does not require $\kappa_h$.

\begin{lemma}[Multi-dimensional initial localization]
\label{lem:multi-initial-localization}
The radius in~\eqref{eq:multi-initial-radius} covers the true parameter with
probability at least $1-\delta/4$ for every initial design. If
\begin{equation}
\label{eq:vector-design-size}
n_0\ge18\frac{d^2\bar x^4}{\kappa_h^2}\log\frac{16d^2}{\delta},
\end{equation}
then, with probability at least $1-\delta/2$,
$\lambda_{\min}(S_0)\ge\frac{n_0\kappa_h}{2}, \qquad \|\widehat{\boldsymbol\alpha}_0-\boldsymbol\alpha\|_2\le r_{0,d}, \qquad \boldsymbol\alpha\in\mathcal A_0,$
and
$\sup_{\boldsymbol\theta\in\mathcal A_0} \|\boldsymbol\theta-\boldsymbol\alpha\|_2 \le\beta_{0,d} \le C_{\mathrm{init},d}\sqrt{\frac{\log(40/\delta)}{n_0}}.$
\end{lemma}

\begin{proof}{Proof.}
Condition on the initial contexts. The scalar conditional Hoeffding
argument in Lemma~\ref{lem:initial-localization} applies in every unit
direction $u$: the OLS error has weights
$u^\top S_0^{-1}(\mathbf X_s-\bar{\mathbf X}_{T_0})$, whose squared sum is $u^\top S_0^{-1}u\le\lambda_{\min}(S_0)^{-1}$. A $1/2$-net of at most $5^d$ directions and the bound
$\|v\|_2\le2\max_u|u^\top v|$ give conditional coverage by the radius
in~\eqref{eq:multi-initial-radius} with failure probability at most $\delta/4$. Projection cannot increase the error; the radius $D_\Theta$ covers the singular-design case.

For the design, apply the same mean-and-second-moment argument entrywise.
A union bound and the operator-norm bound by $d$ times the largest entry
yield, with failure probability at most $\delta/4$, $\|S_0/n_0-\Sigma_h\|_{\mathrm{op}} \le3d\bar x^2\sqrt{\frac{\log(16d^2/\delta)}{2n_0}} \le\kappa_h/2$. Thus~\eqref{eq:vector-design-size} implies
$\lambda_{\min}(S_0)\ge n_0\kappa_h/2$. Intersect with coverage for the stated event. The search-set bound follows from the triangle
inequality and $d\log5+\log(8/\delta)\le d\log(40/\delta)$.
\hfill\Halmos
\end{proof}

For $n_i=|A_i|$, define
\begin{equation}
\label{eq:multi-quantile-radius}
r_{i,d}(\delta)
=
C_q
L\sqrt{\frac{d\log(2n_i)+\log(2L/\delta)}{n_i}}.
\end{equation}
\begin{lemma}[Multi-dimensional estimation error]
\label{lem:multi-estimation-error}
Condition on an initial history on which
Lemma~\ref{lem:multi-initial-localization} holds. Suppose $n_i\ge L$,
$\beta_{0,d}\le r_{\mathrm{id}}$, and
$2\bar x\beta_{0,d}+4r_{i,d}(\delta)\le\min\{\Delta_q,r_g\}, \qquad c_{\mathrm{vc}}L\sqrt{\frac{d\log(2n_i)+\log(2L/\delta)}{n_i}} \le\frac{1-p_0}{4}.$
Under Assumptions~\ref{ass:lipschitz}, \ref{ass:value-density},
\ref{ass:multi-context}, \ref{ass:multi-visible-quantile}, and
\ref{ass:multi-identifiability}, with conditional probability at least $1-\delta$, $\|\widehat{\boldsymbol\alpha}_i-\boldsymbol\alpha\|_2 \le\min\{\beta_{0,d},2r_{i,d}(\delta)/\kappa\} \le C_{\alpha,d}\sqrt{\frac{\log(2n_i/\delta)}{n_i}}.$
Here $C_{\alpha,d}=4C_qL\sqrt d/\kappa$, as
in~\eqref{eq:vector-statistical-factor}. The fallback in
Algorithm~\ref{alg:multi-quantile} is not used on this event.
\end{lemma}

\begin{proof}{Proof.}
Control $L$ empirical value bins uniformly over the vector parameter.
Conditional on the initial history, latent samples remain i.i.d. Intersections of value intervals with residual halfspaces
form a VC class of dimension $O(d)$. Its uniform joint-probability error is $\varepsilon_{i,d}:=C\sqrt{[d\log(2n_i)+\log(2L/\delta)]/n_i}$. For the empirical cutoffs
$\widehat a_{i,\ell}=V_{(\lceil\ell n_i/L\rceil)}$, the same bound gives
$|F_V(\widehat a_{i,\ell})-\ell/L|\le
C(\varepsilon_{i,d}+n_i^{-1})$.
Hence each empirical bin differs from its population bin by probability
mass at most $C(\varepsilon_{i,d}+n_i^{-1})$. Its empirical mass is $L^{-1}+O(n_i^{-1})$ and its population mass is at least $1/(2L)$ under the stated small-error condition. Dividing the joint-probability
error by these masses gives
$\sup_{\boldsymbol\theta\in\mathcal A_0} \max_{\ell\le L}\sup_y |\widehat F^o_{i,\ell,\boldsymbol\theta}(y) -F^o_{\ell,\boldsymbol\theta}(y)| \le c_{\mathrm{vc}}L \sqrt{\frac{d\log(2n_i)+\log(2L/\delta)}{n_i}} =\frac{g_{\min}}2r_{i,d}(\delta).$
Here $\widehat F^o$ and $F^o$ are the empirical and population uncensored
residual CDFs in the corresponding bins.

Every residual displacement is at most $\bar x\beta_{0,d}$.
The density-neighborhood condition permits the inversion in
the proof of Lemma~\ref{lem:uniform-quantile-concentration}, with uniform quantile error at most
$r_{i,d}(\delta)$. Visibility puts censored residuals strictly below the
relevant order statistic of the latent residuals as in
Lemma~\ref{lem:quantile-invariance};
quantiles agree and no fallback occurs.
As in the scalar estimation proof, empirical minimization now gives
$Q^o(\widehat{\boldsymbol\alpha}_i)\le2r_{i,d}(\delta)$, since centered
root-mean-square dispersion is Lipschitz in the sup norm.
Local identifiability and membership in $\mathcal A_0$ give the stated
minimum bound. Finally, $d\log(2n_i)+\log(2L/\delta)\le2d\log(2n_i/\delta)$ for $n_i\ge L$ gives the envelope $C_{\alpha,d}$.
\hfill\Halmos
\end{proof}

\paragraph{The $q$-based bidding algorithm.}
Use the grids, blocks, confidence allocation, dual update, hard-budget
safeguard, and sequential filtering of Algorithm~\ref{alg:budget-bidding}.
Replace every scalar contextual shift $\widehat\alpha_jx$ by $\langle\widehat{\boldsymbol\alpha}_j,\mathbf x\rangle$. The initial CDF is the vector analogue of~\eqref{eq:initial-G-estimator}.
With $\widetilde b_s=b_s-\langle
\widehat{\boldsymbol\alpha}_{i-1},\mathbf x_s\rangle$ and the corresponding
losing residual $\widetilde d_s$, use the same observable $Y_s^k$ as
in~\eqref{eq:G-estimator} and set
\begin{equation}
\label{eq:multi-G-estimator}
n_i^k=\sum_{s\in U_i}\I\{\widetilde b_s\le\widetilde b^k\},
\qquad
\widehat G_i(\widetilde b^k)
=\frac{\sum_{s\in U_i}\I\{\widetilde b_s\le\widetilde b^k\}Y_s^k}
{n_i^k\vee1}.
\end{equation}
The reward estimate is unchanged, and the cost uses the next-phase shift:
\begin{equation}
\label{eq:multi-reward-cost-estimators}
\widehat r_i(q^m,\widetilde b^k)
=(q^m-\widetilde b^k)\widehat G_i(\widetilde b^k),
\qquad
\widehat c_i(\widetilde b^k;\mathbf x)
=\bigl(\widetilde b^k+
\langle\widehat{\boldsymbol\alpha}_i,\mathbf x\rangle\bigr)
\widehat G_i(\widetilde b^k).
\end{equation}
For $1\le i<n$, take
$\delta_i=\delta/[2(n\vee1)]$ and $\beta_{i,d}=\min\{\beta_{0,d},2r_{i,d}(\delta_i)/\kappa\}$. With the same $N_i^m=\min_{\widetilde b^k\in\mathcal B_{i-1}^m}n_i^k$,
define
\begin{equation}
\label{eq:multi-confidence-radius}
\omega_{i,d}^m=\min\left\{2\bar v,\,
C_\omega\left[
\bar v\sqrt{\frac{L_T}{N_i^m\vee1}}
+(1+\bar v g_{\max})\bar x(\beta_{i-1,d}+\beta_{i,d})
\right]\right\}.
\end{equation}
\begin{lemma}[Vector confidence and one-step comparison]
\label{lem:multi-confidence-comparison}
On the common confidence event, the distribution, proxy-reward, and feasible
cost bounds of Lemma~\ref{lem:confidence-representation} hold with
$\beta_j,\omega_i^m$ replaced by $\beta_{j,d},\omega_{i,d}^m$ and scalar
shifts replaced by inner products. All oracle candidates are retained.
For $i\ge2$, $t\in A_i\cup U_i$, and $b\in[0,v_t]$,
\begin{align}
r(v_t,b,\mathbf x_t)-r(v_t,b_t,\mathbf x_t)
&\le C(1+\bar\lambda)
[\omega_{i-1,d}^{m(t)}+\varepsilon_{\mathrm{grid}}]
\notag\\
&\quad+\lambda_t[c(b,\mathbf x_t)-\widehat c_t].
\label{eq:multi-one-step-comparison}
\end{align}
\end{lemma}
\begin{proof}{Proof.}
The scalar proof applies because
$|\langle\widehat{\boldsymbol\alpha}_j-\boldsymbol\alpha,\mathbf x\rangle|
\le\bar x\beta_{j,d}$ and~\eqref{eq:multi-G-estimator} has the same
observable representation. Feasible bids remain at most $\bar v$, so
both contextual shifts in~\eqref{eq:multi-confidence-radius} cover the
cost error. The proxy objective $r_q(q,\widetilde b)$ has no contextual
coordinate; its oracle monotonicity and retention are unchanged.
The identity~\eqref{eq:fixed-context-lagrangian} with $\mathbf x$ fixed
gives the one-step comparison.
\hfill\Halmos
\end{proof}

\begin{lemma}[Multi-dimensional confidence sum]
\label{lem:multi-confidence-sum}
Under the assumptions and horizon conditions of
Proposition~\ref{prop:multi-dimensional-guarantee},
$\mathbb E\!\left[ \sum_{i=2}^n\sum_{t\in A_i\cup U_i}\omega_{i-1,d}^{m(t)} \right] \le C\mathcal C_{\mathrm{stat},d}\sqrt T\log T.$
\end{lemma}

\begin{proof}{Proof.}
The vector effective-value margin and dual bound give
$\gamma_qv_t\le q_t\le v_t$. Oracle retention in
Lemma~\ref{lem:multi-confidence-comparison} makes the smallest active bid
feasible. Hence $v_s\le\gamma_qv_t$ implies $q_s\le q_t$; the scalar rank lemmas
in Appendix~\ref{app:confidence-sum-bound} apply with the same unconditional
value calculation.
The sampling contribution is consequently at most $C\bar v\sqrt{TL_T/\kappa_V}$. 
The localization and estimation lemmas give
$\beta_{0,d}\le C_{\mathrm{init},d}
\sqrt{\log(40/\delta)/n_0}$ and
$\beta_{i,d}\le C_{\alpha,d}
\sqrt{\log(2|A_i|/\delta_i)/|A_i|}$.
The same geometric phase sum bounds
$\sum_{i=2}^n|A_i\cup U_i|
(\beta_{i-2,d}+\beta_{i-1,d})$
by $C(C_{\mathrm{init},d}+C_{\alpha,d})\sqrt{T\ell_T}$,
where $\ell_T=\log(8T^4)$. Substitute in the radius; its cap bounds the expected
failure contribution by $C\bar v/T$.
\hfill\Halmos
\end{proof}

\paragraph{Proof of Proposition~\ref{prop:multi-dimensional-guarantee}.}
\begin{proof}{}
Set $\delta=T^{-2}$, $\ell_T=\log(8T^4)$,
$\overline\beta_{0,d}(T)=C_{\mathrm{init},d}
\sqrt{\log(40/\delta)/n_0}$, and $\overline r_d(T)=C_qL\sqrt{2d\ell_T/H_0}$. A sufficient set of horizon conditions is
\begin{equation}
\label{eq:multi-large-horizon}
\begin{gathered}
T\ge\max\{16,(\bar v/\rho)^2\},\qquad H_0\ge L,\qquad
n_0\ge18\frac{d^2\bar x^4}{\kappa_h^2}\log\frac{16d^2}{\delta},\\
\overline\beta_{0,d}(T)\le r_{\mathrm{id}},\qquad
2\bar x\overline\beta_{0,d}(T)+4\overline r_d(T)
\le\min\{\Delta_q,r_g\},\\
c_{\mathrm{vc}}L\sqrt{\frac{2d\ell_T}{H_0}}\le\frac{1-p_0}{4}.
\end{gathered}
\end{equation}
These conditions ensure localization, nonsingular design, nonempty bins,
and quantile-neighborhood control. Every used block has
size at least $H_0$, and $\log(2|A_i|/\delta_i)\le\ell_T$. The preceding lemmas and the same confidence allocation therefore give a
common event $\mathcal E_d$ with
$\mathbb P(\mathcal E_d^c)\le2\delta$.
On this event, fallback estimates and empty-set resets are unused.
For fixed parameters, the conditions hold for all
sufficiently large $T$.

Apply the proof of Theorem~\ref{thm:main-regret}, replacing its confidence
and confidence-sum lemmas by
Lemmas~\ref{lem:multi-confidence-comparison}
and~\ref{lem:multi-confidence-sum}.
The relaxed stationary benchmark and fixed-context identity carry over
with the observed vector unchanged. Feasibility still gives
$0\le\widehat c_t\le b_t\le\bar v/(1+\lambda_t)$, so
Lemmas~\ref{lem:dual-bound}, \ref{lem:c-bound},
and~\ref{lem:stopping-time-bound} apply verbatim.
The scalar initial, grid, and failure bounds yield
$\operatorname{Regret}(\pi) \le C(1+\bar\lambda)\mathcal C_{\mathrm{stat},d}\sqrt T\log T +C\eta T\bar v^2+\frac{\bar v\bar\lambda}{\rho\eta}.$
Under~\eqref{eq:multi-large-horizon},
$\eta=(\rho\sqrt T)^{-1}$. The remaining terms are $O(\bar v^2\sqrt T/\rho)$, absorbed by the
statistical factor.
\hfill\Halmos
\end{proof}

\end{APPENDICES}

\end{document}